\documentclass{amsart}

\usepackage{amsmath,amsfonts,amssymb,amscd,comment,euscript, mathtools}
\usepackage{etoolbox}
\usepackage{stmaryrd}
\usepackage[all]{xy}
\usepackage{graphicx}
\usepackage{enumerate, enumitem} 
\usepackage[breaklinks, colorlinks=true, linktocpage]{hyperref} 
\usepackage{pifont}
\usepackage[usenames,dvipsnames]{xcolor}
\usepackage{tikz-cd, rotating}
\usepackage{tikz}
\usetikzlibrary{arrows.meta}
\usetikzlibrary{decorations.markings}
\usetikzlibrary{calc}
\usepackage{xcolor}
\usepackage{soul}
\usepackage{stmaryrd} 

\usepackage[maxbibnames=9,backend=bibtex]{biblatex}
\usepackage{amsthm}
\usepackage{thmtools}
\usepackage[capitalize]{cleveref}   

\renewcommand\rightmark{\footnotesize Algebraic cobordism rings of wonderful varieties and matroids}

\theoremstyle{plain}
\newtheorem{lem}{Lemma}[section]
\newtheorem{cor}[lem]{Corollary}
\newtheorem{prop}[lem]{Proposition}

\newtheorem{theo}[lem]{Theorem}

\theoremstyle{definition}
\newtheorem{egg}[lem]{Example}
\newtheorem{rmk}[lem]{Remark}
\newtheorem{defn}[lem]{Definition}

\newcommand{\arxiv}[1]{\href{http://arxiv.org/abs/#1}{\tt arXiv:\nolinkurl{#1}}}
\newcommand{\arxivpdf}[1]{\href{http://arxiv.org/pdf/#1}{\tt arXiv:\nolinkurl{#1}}}

\newcommand{\mrm}{\mathrm}

\newcommand{\mbb}{\mathbb}
\newcommand{\mcal}{\mathcal}
\newcommand{\CH}{\mathrm{CH}}
\newcommand{\cone}{\mathrm{cone}}

\begin{document}

\title[]{Algebraic cobordism rings of wonderful varieties and matroids}
\author[]{Raj Gandhi and Ethan Partida}

\address{
    Department of Mathematics \\
    Cornell University
}
\email{rg593@cornell.edu}

\address{
    Department of Mathematics \\
    Brown University
}
\email{ethan\_partida@brown.edu}

\begin{abstract}
    We give two combinatorial presentations for the algebraic cobordism ring $\Omega^*(M)$ of the toric variety of the Bergman fan of any loopless matroid $M$. As a consequence of our presentations, we obtain an $\Omega^*(\mrm{pt})$-algebra isomorphism $\Omega^*(M) \simeq \CH^*(M) \otimes_{\mbb{Z}}  \Omega^*(\mrm{pt})$, where $\CH^*(M)$ is the Chow ring of $M$ and $\Omega^*(\mrm{pt})$ is the algebraic cobordism ring of the point. This isomorphism generalizes, in part, the exceptional integral isomorphism between the Chow ring and $K$-ring of a matroid, studied in the recent works of Berget--Eur--Spink--Tseng and Larson--Li--Payne--Proudfoot. For a complex hyperplane arrangement $\mcal{H}$, we prove that the algebraic cobordism ring of the wonderful variety $W_\mcal{H}$ of $\mcal{H}$ and the algebraic cobordism ring of the toric variety of the matroid underlying $\mcal{H}$ are isomorphic, and that both rings coincide with the complex cobordism ring of $W_\mcal{H}$.
\end{abstract}
	
\maketitle
\tableofcontents

\section{Introduction}
    Let $\mcal{H}$ be an arrangement of hyperplanes in a complex vector space $L$ whose intersection is the origin in $L$. The arrangement $\mcal{H}$ gives rise to a loopless matroid $M_\mcal{H}$ whose flats are indexed by the various intersections of hyperplanes in $\mcal{H}$. De Concini and Procesi defined a smooth compactification $W_{\mcal{H}}$ of the complement of the hyperplane arrangement in $L$, known as the \textit{wonderful compactification} of $\mcal{H}$ \cite{DP95}, whose boundary is a simple normal crossings divisor. Much of the geometry of $W_{\mcal{H}}$ is dictated by invariants of the matroid $M_{\mcal{H}}$. As $W_{\mcal{H}}$ is a smooth, projective variety, such invariants of $M_{\mcal{H}}$ must satisfy strong structural properties. Sometimes these properties can be given a purely combinatorial explanation. This, in turn, lets one prove new structural results about matroids, even those not realizable over any field.

    Given a  loopless matroid $M$, there is a (usually non-compact) toric variety $X_M$ canonically associated to $M$ whose Chow ring $\CH^*(X_M)$ has an explicit description \cite{FY04}. Feichtner and Yuzvinsky proved that when $M_{\mcal{H}}$ is the matroid of a complex hyperplane arrangement $\mcal{H}$, there is a canonical inclusion $W_{\mcal{H}} \hookrightarrow X_{M_\mcal{H}}$ which induces an isomorphism of Chow rings $\CH^*(X_{M_\mcal{H}})\simeq \CH^*(W_\mcal{H})$ \cite{FY04}. As the cycle class map is an isomorphism for $W_{\mcal{H}}$ (Lemma \ref{lem:cycle-class}), this gives the Chow ring of $X_{M_\mcal{H}}$ the structure of the cohomology ring of a smooth, complex, projective variety. An important consequence of this is that $\CH^*(X_{M_{\mcal{H}}})$ has the K\"ahler package \cite{EL}. Remarkably, Adiprasito--Huh--Katz proved that $\CH^*(X_{M})$ has the K\"ahler package for every loopless matroid $M$, regardless of whether or not $M$ is realizable \cite{AHK18}. This result is a key step in their proof of the Heron--Rota--Welsh conjecture \cite{AHK18}.

    In recent work \cite{LLPP24}, Larson--Li--Payne--Proudfoot studied the Grothendieck ring of algebraic vector bundles on $W_{\mcal{H}}$, which we denote by $K^0(W_{\mcal{H}})$ (following the notation of Levine and Morel \cite{LM07}), and refer to as the $K$-ring  of $W_{\mcal{H}}$. They proved that the inclusion $W_{\mcal{H}} \hookrightarrow X_{M_{\mcal{H}}}$ induces an isomorphism of $K$-rings $K^0(X_{M_\mcal{H}})\simeq K^0(W_\mcal{H})$. The authors also proved that, for any loopless matroid $M$, the $K$-ring of $X_M$ admits structures such as Adams operations, an Euler characteristic map, and Serre duality. Surprisingly, they also showed that the Chow ring and $K$-ring of $X_M$ are integrally isomorphic. This isomorphism is known as an \emph{exceptional isomorphism} in the literature and was first proven in the Boolean case by Berget--Eur--Spink--Tseng \cite{BEST}. The exceptional isomorphism is not directly related to the Chern character isomorphism between the rational $K$-ring and the rational Chow ring of $W_\mcal{H}$.

    The functors $\CH^*(-)$ and $K^0(-)$ which send a smooth complex variety to its Chow ring and $K$-ring are examples of \emph{oriented cohomology theories} (\Cref{subsection:cobordism}). There is a universal oriented cohomology theory $\Omega^*(-)$ called \emph{algebraic cobordism} constructed by Levine and Morel \cite{LM07}. Levine and Morel's result is an analogue of the fundamental observation of Quillen that complex cobordism $MU^*(-)$ is the universal complex-oriented cohomology theory \cite{Q71}. For a smooth complex variety $X$, there is a natural morphism $\Omega^*(X)\to MU^*(X)$ which is an isomorphism when $X$ is a point, but is not an isomorphism in general (\Cref{rmk:cobordismComparison}). For any free oriented cohomology theory $h^*(-)$, of which $\CH^*(-)$ and $K^0(-)$ are examples, the natural morphism $\Omega^*(X)\to h^*(X)$ is surjective and $h^*(X)$ can be recovered from $\Omega^*(X)$ (\Cref{dfn:free-theory}). Computations of algebraic cobordism rings are generally very difficult and explicit ring presentations are uncommon. Some of the few families of varieties for which presentations of their algebraic cobordism rings are known are smooth toric varieties \cite{KU13}, complete flag varieties \cite{HK11, CPZ}, Pﬁster quadrics \cite{VY7}, and wonderful symmetric varieties of minimal rank \cite{KK13}. We refer the reader to Section \ref{section:formal-group-laws-and-algebraic-cobordism} for an introduction to algebraic cobordism. 
    
    In this paper, we study the algebraic cobordism rings of wonderful compactifications $W_{\mcal{H}}$ and toric varieties of Bergman fans of loopless matroids $X_M$. Our first result is a ``simplicial'' presentation of the ring $\Omega^*(X_M)$. In this presentation, the generators of $\Omega^*(X_M)$ are lifts of the ``simplicial generators" for $\CH^*(X_M)$, introduced by Backman--Eur--Simpson \cite{BES24}, to the algebraic cobordism ring of $X_M$. We derive the simplicial presentation from careful manipulations of the universal formal group laws that appear in the presentation of $\Omega^*(X_M)$ deduced from Krishna and Uma's description of the algebraic cobordism ring of a smooth toric variety \cite{KU13}. In the following, we use $\mrm{pt}= \operatorname{Spec}(\mbb{C})$ and refer the reader to Sections \ref{section:wonderful-varieties} and \ref{section:formal-group-laws-and-algebraic-cobordism} for undefined terminology. 

    \begin{theo}\label{theo:intro1}(\cref{thrm:simplicial})
        Let $M$ be a loopless matroid. The algebraic cobordism ring of $X_M$ is isomorphic to
        \[\frac{\Omega^*(\mrm{pt})[h_G: \text{ $G$ a non-empty flat of $M$}]}{\left\langle (h_G-h_{G\vee H})(h_H-h_{G\vee H}): G,H \text{ non-empty flats of $M$}\right\rangle+\left\langle h_G: G \text{ atom}\right\rangle }  \]
        as an $\Omega^*(\mrm{pt})$-algebra. This presentation yields an isomorphism of $\Omega^*(\mrm{pt})$-algebras:
          \begin{equation}\label{eqn:simplicial}
        \Omega^*(X_M)\simeq \mrm{CH}^*(X_M)\otimes_{\mbb{Z}} \Omega^*(\mrm{pt}).
        \end{equation}
    \end{theo}

    The isomorphism (\ref{eqn:simplicial}) is unexpected, and it immediately implies that there is an integral ring isomorphism $K^0(X_M)\xrightarrow{\sim} \CH^*(X_M)$ (\Cref{cor:exceptional-explanation}). In fact, for any free oriented cohomology theory $\mrm{h}^*(-)$, we have that $\mrm{h}^*(X_M)$ is isomorphic to $\Omega^*(X_M) \otimes_{\Omega^*(\mrm{pt})} \mrm{h}^*(\mrm{pt})$ as an $\mrm{h}^*(\mrm{pt})$-algebra. Thus, the isomorphism (\ref{eqn:simplicial}) can be viewed as a generalization of the exceptional isomorphism $K^0(X_M)\xrightarrow{\sim} \CH^*(X_M)$. The only other varieties $X$ that the authors are aware of which have an $\Omega^*(\mrm{pt})$-isomorphism $\Omega^*(X)\simeq \mrm{CH}^*(X)\otimes_{\mbb{Z}}\Omega^*(\mrm{pt})$ are projective space \cite{LM07}, del Pezzo surfaces \cite[Proposition 4.1]{DZ26}, and complete flag varieties \cite[Theorem 1.1]{HK11}. Even if $X$ is a smooth, projective toric variety, there can fail to be such an isomorphism. Indeed, Speyer \cite{Speyer} gives an example of a smooth, projective toric variety $X$ such that $\mrm{CH}^*(X)$ and $K^0(X)$ are not isomorphic as rings, and from this it follows that $\Omega^*(X)$ and $\mrm{CH}^*(X)\otimes_{\mbb{Z}}\Omega^*(\mrm{pt})$ are not isomorphic as $\Omega^*(\mrm{pt})$-algebras. This example suggests that the isomorphism (\ref{eqn:simplicial}) is very special. It would be interesting to have a more systematic study of the varieties $X$ for which $\Omega^*(X)\simeq \mrm{CH}^*(X)\otimes_{\mbb{Z}}\Omega^*(\mrm{pt})$. For completeness of our exposition, we have included Speyer's example as \cref{egg:smooth-projective-toric} in this article. 

    As our next main result, we prove that the inclusion $W_{\mcal{H}} \hookrightarrow X_{M_{\mcal{H}}}$ induces an isomorphism of algebraic cobordism rings $\Omega^*(X_{M_\mcal{H}})\simeq \Omega^*(W_\mcal{H})$.
    \begin{theo}\label{theo:intro2}(\cref{theo:pullback-in-cobordism})
        Let $W_{\mcal{H}}$ be the wonderful compactification of a hyperplane arrangement $\mcal{H}$ and $M_{\mcal{H}}$ its associated matroid. The inclusion $\iota: W_{\mcal{H}} \hookrightarrow X_{M_{\mcal{H}}}$ induces an isomorphism of $\Omega^*(\mrm{pt})$-algebras:
        \[\iota^*: \Omega^*(X_{M_{\mcal{H}}})\xrightarrow{\sim} \Omega^*(W_{\mcal{H}}).  \]
    \end{theo}

    Along the way to proving \Cref{theo:intro2}, we prove that $\Omega^*(W_{\mcal{H}})$ is isomorphic, as an $\Omega^*(\mrm{pt})$-algebra, to the complex cobordism ring $MU^*(W_{\mcal{H}})$ of $W_{\mcal{H}}$ (\Cref{theo:pullback-in-cobordism}). Comparisons between the complex and algebraic cobordism rings of $W_{\mcal{H}}$ are a key technical aspect of our proof of \Cref{theo:intro2}. 
    
    In a very recent preprint, Debnath and Zeng give a presentation of the oriented cohomology ring of the blowup of a smooth scheme along a smooth center \cite{DZ26}. As $W_\mcal{H}$ is an iterated blowup of projective space and $\Omega^*(-)$ is an oriented cohomology theory, it should be possible to derive a presentation of the algebraic cobordism ring of $W_H$ from Debnath and Zeng's result. However, such a ring presentation will likely be a few steps removed from the presentation of $\Omega^*(W_\mcal{H})$ given by Theorems \ref{theo:intro1} and \ref{theo:intro2}. In fact, the analogous translation between two different presentations of the Chow ring of $W_\mcal{H}$ is one of the main results of \cite{FY04}.
    
    We close by noting that the class $[W_{\mcal{H}}\to \mrm{pt}]$ of the wonderful compactification in the algebraic (and also complex) cobordism ring of a point only depends on the matroid $M_{\mcal{H}}$ of $\mcal{H}$. This follows from results of Cheng on the Chern numbers of the tangent bundle of $W_{\mcal{H}}$ \cite{Cheng25}. In future work, the first author plans to give a recursive formula for the class of $W_{\mcal{H}}$ in $\Omega^*(\mrm{pt})$ using recent techniques of Ellis-Bloor \cite{EB26}, who gives a recursive formula for the class of $\overline{\mcal{M}_{0,n}}$ in $\Omega^*(\mrm{pt})$. Obtaining such a formula can be seen as the first step towards explicitly describing the proper push-forward map $\Omega^*(W_\mcal{H})\to \Omega^*(\mrm{pt})$. See \cref{rmk:future-work} for further details. 

    This paper is organized as follows. In \cref{section:wonderful-varieties}, we recall matroids, wonderful compactifications, and Bergman fans. In \cref{section:formal-group-laws-and-algebraic-cobordism}, we recall the theory of formal group laws and oriented cohomology theories.  In \cref{section:algebraic-cobordism-XM}, we describe the simplicial presentation of $\Omega^*(X_M)$ and prove \Cref{theo:intro1}. In \cref{cobordism-of-wonderful}, we prove \Cref{theo:intro2} and study the class of $[W_{\mcal{H}}\to \mrm{pt}]$ in both the algebraic and complex cobordism rings of a point.

    \subsection*{Acknowledgments}
    We thank Kenneth Blakey, Melody Chan, Liam Keenan, Carly Klivans, Allen Knutson, Matt Larson, Marc Levine, Shiyue Li, Ruizhen Liu, Anubhav Nanavaty, and Burt Totaro for helpful conversations and correspondences, and we thank the anonymous referees for their remarks on an extended abstract of this document submitted to the conference ``Formal Power Series and Algebraic Combinatorics". We thank Sam Payne for pointing out an error with \cref{egg:smooth-projective-toric} in an earlier draft of this paper. We also thank David E Speyer for discussing \cref{egg:smooth-projective-toric} with us. R.G. was partially supported by a Postgraduate Scholarship - Doctoral - from the Natural Sciences and Engineering Research Council of Canada. E.P. was partially supported by NSF Grant DMS-2053288, a U.S. Department of Education GAANN award, and the Simons Foundation SFI-MPS-SDF-00015018.

\section{Recollection of wonderful varieties and matroids}\label{section:wonderful-varieties}
    In \cref{subsection:matroids-and-hyperplane-arrangements}, we recall matroids and hyperplane arrangements. In \cref{subsection:wonderful}, we recall the wonderful compactifications of De Concini-Procesi. In \cref{subsection:bergman}, we introduce the conventions for fans and toric varieties that we will follow in this paper and recall the Bergman fan of a matroid.
    
\subsection{Matroids and hyperplane arrangements}\label{subsection:matroids-and-hyperplane-arrangements}
    We will recall the definition of a matroid, and describe the matroid corresponding to a hyperplane arrangement.

    \begin{defn}
        Let $E$ be a finite set. A loopless \textbf{matroid} $M$ on the ground set $E$ is a collection $\mcal{F}$ of subsets of $E$ called \textbf{flats} that satisfy:
        \begin{itemize}
            \item $\emptyset\in \mcal{F}$.
            \item If $G,H\in \mcal{F}$, then $G\cap H\in\mcal{F}$.
            \item For any $G\in\mcal{F}$ and any $i\in E\setminus G$, there is a unique, inclusion-minimal flat $H\in \mcal{F}$ containing $i$ and $G$.
        \end{itemize}
        For two flats $G$ and $H$, their \textbf{meet} $G\wedge H$ is the intersection $G\cap H$ and their \textbf{join} $G\vee H$ is the smallest flat containing $G\cup H$. Two flats $H$ and $G$ are called \textbf{incomparable} if $H\not\subseteq G$ and $G\not\subseteq H$. If $H$ and $G$ are not incomparable, then they are \textbf{comparable}.
        When ordered by inclusion, the set $\mcal{F}$ forms a geometric lattice.  The \textbf{rank} $\mrm{rk}(F)$ of a flat $F$ is equal to the length of the largest chain of comparable  flats properly contained in $F$.  An \textbf{atom} is a flat of rank $1$. It follows from the axioms that $E\in \mcal{F}$. Let $r$ be the rank of $E$. In this case, we say that $M$ has rank $r$. The \textbf{corank} of a flat $G$ of $M$ is the number $r-\mrm{rk}(G)$.
    \end{defn}

   \begin{defn}
      A complex, central and essential \textbf{hyperplane arrangement} $\mcal{H}$ is a multiset of hyperplanes $\{H_i\}_{i\in E}$ in a complex vector space $L$ indexed by a finite set $E$ such that $\bigcap_{i\in E} H_i =\{0\}$. 
   \end{defn}

    In this paper, all hyperplane arrangements will be complex, central and essential. We henceforth drop the extra adjectives. 

    Every hyperplane arrangement $\mcal{H}=\{H_i\}_{i\in E}$ determines a loopless matroid $M_{\mcal{H}}$ on the ground set $E$. For a subset $S\subseteq E$, let $L_S$ be the linear subspace $L_S = \bigcap_{i\in S} H_i$. The set of flats of $M_{\mcal{H}}$ is equal to
    \[\mcal{F}_{\mcal{H}} = \{ F \subseteq E : \text{ if } i\in E\setminus F, \text{ then } L_{F\cup i} \neq L_F\}. \]
    In this case, the rank of a flat $F$ is equal to the codimension of $L_F$ in $L$. 

    For an integer $n\geq 1$, let $[n] = \{1,2,\ldots, n\}$.
    \begin{egg}
        Let $\mcal{H}$ be an arrangement of generic hyperplanes indexed by a finite set $E$ in an $r$-dimensional vector space $L$. The flats of $M_\mcal{H}$ consist of all subsets $S\subseteq E$ with size strictly less than $r$ along with the set $E$. In this case, we refer to $M_\mcal{H}$ as the \textbf{uniform matroid} of rank $r$ on the ground set $E$. If $r=|E|$, then we say that $M_\mcal{H}$ is the \textbf{Boolean matroid} on the ground set $E$. 
        When $E=[n]$, then we denote $M_{\mcal{H}}$ by $U_{r,n}$.
    \end{egg}
    
    \begin{egg}\label{egg:U23}
        Consider the hyperplane arrangement $\mcal{H}=\{H_1,H_2,H_3\}$ in $\mbb{C}^2$, where $H_1=\mrm{span}_{\mbb{C}}\{(1,-1)\}$, $H_2=\mrm{span}_{\mbb{C}}\{(1,1)\}$, and $H_3=\mrm{span}_{\mbb{C}}\{(1,0)\}$. The matroid underlying this arrangement is $M_{\mcal{H}}=U_{2,3}$. This matroid has a single flat of rank $2$, corresponding to the intersection $H_1\cap H_2\cap H_3$, which is the origin in $\mbb{C}^2$. There are three flats of rank $1$, which correspond to the hyperplanes $H_1$, $H_2$, and $H_3$. Finally, there is a single flat of rank $0$, corresponding to the empty intersection, i.e., the ambient vector space $\mbb{C}^2$. We illustrate this example in Figure \ref{fig:hyperplane-arrangement} below.
        
        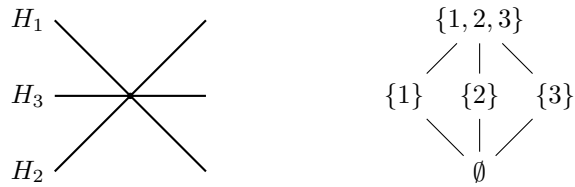
\begin{figure}[h]
        \begin{tikzpicture}[scale=0.5]
            \draw[black, thick] (-2, 0) node[left] {$H_3$} -- (2, 0) ;
            \draw[black, thick] (-2, -2) node[left] {$H_2$}-- (2, 2) ;
            \draw[black, thick] (-2, 2) node[left] {$H_1$} -- (2, -2)  ;
            \filldraw[black] (0,0) circle (2pt);
        \end{tikzpicture}  \quad\quad\quad\quad\quad\quad \begin{tikzpicture}[scale=.5]
          \node (one) at (0,2) {$\{1,2,3\}$};
          \node (a) at (-2,0) {$\{1\}$};
          \node (b) at (0,0) {$\{2\}$};
          \node (c) at (2,0) {$\{3\}$};
          \node (zero) at (0,-2) {$\emptyset$};
          \draw (zero) -- (a) -- (one) -- (b) -- (zero) -- (c) -- (one);
        \end{tikzpicture}\caption{On the left, a hyperplane arrangement whose corresponding matroid is $U_{2,3}$, and on the right, the lattice of flats of $U_{2,3}$.
        }\label{fig:hyperplane-arrangement}
        \end{figure}
    \end{egg}

\subsection{Wonderful varieties of hyperplane arrangements}\label{subsection:wonderful}
    We recall the definition of the wonderful compactification of a hyperplane arrangement $\mcal{H}$ with respect to the maximal building set. Wonderful compactifications were introduced in the work of de Concini and Procesi \cite{DP95}. There, they define wonderful compactifications with respect to any choice of \emph{building set}. A building set is a subset of the flats of $M_{\mcal{H}}$ which obeys certain axioms. The set $\mcal{F}_{\mcal{H}}\setminus \emptyset$ is always a building set and is known as the maximal building set. Wonderful compactifications, as we define them here, will always be with respect to the maximal building set.
    
    Let $\mcal{H}$ be a hyperplane arrangement in a complex vector space $L$ whose hyperplanes are indexed by a finite set $E$. Now projectivise each hyperplane $H_i$ in $\mcal{H}$ to obtain a collection of hyperplanes $\mbb{P}\mcal{H}= \{\mbb{P}(H_i)\}_{i\in E}$ inside the projective space $\mbb{P}(L)$. The wonderful variety $W_{\mcal{H}}$ is a smooth compactification of the open subset $\mbb{P}(L)\setminus \bigcup_{i\in E}\mbb{P}(H_i)$ of $\mbb{P}(L)$, whose boundary is a simple normal crossings divisor. We will now discuss the construction of $W_\mcal{H}$. 

    \begin{defn}
        The \textbf{wonderful variety} $W_{\mcal{H}}$ of a hyperplane arrangement $\mcal{H}$ can be constructed as an iterated blowup of $\mbb{P}(L)$. 
        \begin{enumerate}
            \item First, blow up $\mbb{P}(L)$ at $\mbb{P}(L_G)$ for all corank $1$ flats $G$ in $M_\mcal{H}$. Call the resulting variety $W_{\mcal{H}}^{(1)}$.
            \item Next, blow up $W_{\mcal{H}}^{(1)}$ at the strict transforms of $\mbb{P}(L_G)$ for all corank $2$ flats $G$ in $M_\mcal{H}$. Call the resulting variety $W_{\mcal{H}}^{(2)}$.
            \item Continue in this way to obtain $W_{\mcal{H}}^{(k)}$ for each $k=1,\ldots,r-1$. This results in a sequence of blow-ups
            \[\pi \colon W_{\mcal{H}}^{(r-1)}\to W_{\mcal{H}}^{(r-2)}\to \dotsb W_{\mcal{H}}^{(2)}\to W_{\mcal{H}}^{(1)}\to \mbb{P}(L)\]
            and we define $W_{\mcal{H}}=W_{\mcal{H}}^{(r-1)}$.
            \end{enumerate}  
    \end{defn}

    \begin{theo}(\cite{DP95})
        The iterated blow-up $\pi\colon W_{\mcal{H}}\to \mbb{P}(L)$ is a smooth compactification of $\mbb{P}(L)\setminus \bigcup_{i\in E}\mbb{P}(H_i)$ whose boundary is a simple normal crossings divisor. The irreducible components of this divisor are indexed by the proper non-empty flats of $M_{\mcal{H}}$.
    \end{theo}

    \begin{defn}\label{defn:divisor-in-W}
        For each flat $G$ of  $M_{\mcal{H}}$, let $U_G$ be the open locus $U_G=\mbb{P}(L_G)\setminus \bigcup_{G\subsetneq H\neq E}\mbb{P}(L_H)$ and define the divisor $D_G$ to be the closure of the preimage $\pi^{-1}(U_G)$ in $W_\mcal{H}$.
    \end{defn}

    \begin{egg}\label{egg:U23wonderful}
        Consider the hyperplane arrangement $\mcal{H}$ of \cref{egg:U23}, whose underlying matroid is $M_\mcal{H}=U_{2,3}$. We have $L=\mbb{C}^2$, and $\mbb{P}(L)=\mbb{P}^1$. The hyperplanes $\mbb{P}(H_1)$, $\mbb{P}(H_2)$, and $\mbb{P}(H_3)$ are three distinct points in $\mbb{P}(L)$. Observe that the blow-up of $\mbb{P}^1$ at a point is isomorphic to $\mbb{P}^1$. As $W_{\mcal{H}}$ is the blow-up of $\mbb{P}^1$ at three distinct points, it follows that $W_\mcal{H}\simeq \mbb{P}^1$. In this case, $W_\mcal{H}\simeq \mbb{P}^1$ is a smooth compactification of $\mbb{P}(L)\setminus\{\mbb{P}(H_1),\mbb{P}(H_2),\mbb{P}(H_3)\}$, whose boundary is the simple normal crossings divisor $D_{G_1}+D_{G_2}+D_{G_3}$, where $D_{G_i}=\mbb{P}(H_i)$. We illustrate $W_\mcal{H}$, together with the boundary divisor, in Figure \ref{fig:wonderful} below.
        \begin{figure}[h]
        \begin{tikzpicture}[scale=0.5]
        \draw[black, thick] (-8, 0) -- (8, 0) ;
        \filldraw[black] (-4,0) circle (2pt);
        \node[above] at (-4,0) {$\mbb{P}(H_1)$} ;
        \filldraw[black] (0,0) circle (2pt) ;
         \node[above] at (0,0) {$\mbb{P}(H_2)$} ;
        \filldraw[black] (4,0) circle (2pt) ;
         \node[above] at (4,0) {$\mbb{P}(H_3)$} ;
        \end{tikzpicture}
        \caption{The wonderful variety of the hyperplane arrangement $\mcal{H}$ from \cref{egg:U23}.}\label{fig:wonderful}
        \end{figure}
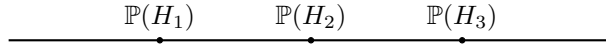
    \end{egg}

\subsection{Toric varieties and Bergman fans}\label{subsection:bergman}
    Let $N$ be a lattice of rank $n$ and $N_{\mbb{R}}$ be the $n$-dimensional vector space $N_{\mbb{R}}\coloneqq N \otimes_\mbb{Z} \mbb{R}$.  A $k$-dimensional, rational, simplicial cone  $\sigma \subset N_{\mbb{R}}$ is a subset that can be written as 
    \[\sigma= \cone\{v_1,\ldots, v_k\} \coloneqq \{\lambda_1v_1+\ldots + \lambda_k v_k : \lambda_i\geq 0 \}\subset N_{\mbb{R}}\] 
    where $v_1,\ldots, v_k$ are $k$ linearly independent vectors in $N$. A cone $\tau$ is a \textbf{face} of $\sigma$ if it can be written as $\tau = \cone\{v_{i_1},\ldots, v_{i_\ell}\}$ for some subset $\{v_{i_1},\ldots, v_{i_\ell}\} \subset \{v_1,\ldots, v_k\}$. We declare $\cone(\emptyset)= \{0\}$ so that $\{0\}$ is a face of every cone. A \textbf{rational, simplicial fan} $\Sigma$ in $N_{\mbb{R}}$ is a non-empty collection of rational, simplicial cones $\sigma \in \Sigma$ in $N_{\mbb{R}}$ such that
    \begin{itemize}
    \item If $\sigma \in \Sigma$ and $\tau$ is a face of $\sigma$, then $\tau \in \Sigma$, and
    \item Every two cones $\sigma, \tau \in \Sigma$ intersect along a common face.
    \end{itemize}
    We say that $\Sigma$ is \textbf{complete} if the union of all of its cones is equal to the whole vector space $N_{\mbb{R}}$. Let $\Sigma(1)$ denote the set of rays, i.e. one dimensional cones, of $\Sigma$. For a ray $\rho\in \Sigma(1)$, the \textbf{primitive ray generator} $\nu_\rho\in N$ of $\rho$ is the generator of the semigroup $N\cap \rho$. We say that $\Sigma$ is \textbf{unimodular} if, for all cones $\sigma\in \Sigma$, the set of primitive ray generators of $\sigma$ can be extended to a basis of $N$.

    Let $T$ be an $n$-dimensional complex torus, and let $N=\mrm{Hom}(\mbb{C}^*,T)$ be its lattice of one-parameter subgroups. Denote the lattice of characters $\mrm{Hom}(T,\mbb{C}^*)$ of $T$ by $N^\vee$. A rational fan $\Sigma$ in $N_{\mathbb{R}}$ gives rise to a toric variety $X_\Sigma$ with torus $T$. There is a rich dictionary between the combinatorics of $\Sigma$ and the geometry of $X_\Sigma$. We refer the reader to the textbook \cite{fultonToric} for details.

    Let $M$ be a loopless matroid on ground set $E$ and $\mbb{Z}^E$ be the lattice with basis $\{e_i : i \in E\}$. For a subset $G\subset E$, define $e_G = \sum_{i\in G} e_i \in \mathbb{Z}^E$. Let $N_E$ be the lattice $\mbb{Z}^E/\mbb{Z}e_E$ and $\overline{e}_G$ denote the image of $e_G$ in $N_E$.

    \begin{defn}\label{defn:bergman-fan}
        The \textbf{Bergman fan} $\Sigma_M$ of $M$ is the unimodular fan in $\mathbb{R}^E/\mathbb{R} e_E = N_E \otimes_\mbb{Z} \mbb{R}$ with cones
        \[\sigma_{\mcal{G}} = \text{cone}\{\overline{e}_G\}_{G\in \mathcal{G}}, \quad \text{for $\mathcal{G}$ a flag of proper and non-empty flats of $M$.} \]
        We will denote the toric variety of $\Sigma_M$ by $X_M$.
    \end{defn}

    The fan $\Sigma_M$ is a  unimodular fan and hence the toric variety $X_M$ is smooth. However, except when $M$ is the Boolean matroid, the fan $\Sigma_M$ is not complete and $X_{M}$ is not proper.
    \begin{egg}\label{egg:U23bergman}
        Let $e_1,e_2,e_3$ be the standard basis vectors in $\mbb{R}^3$. The flags of flats for $U_{2,3}$ are 
        \[\emptyset\subseteq \{1\}\subseteq \{1,2,3\},\quad \emptyset\subseteq \{2\}\subseteq \{1,2,3\},\quad \emptyset\subseteq \{3\}\subseteq \{1,2,3\}.\]
        Therefore, the Bergman fan of $U_{2,3}$ lies in $\mbb{R}^3/(e_1+e_2+e_3)$ and has cones 
        \[\sigma_1=\mrm{span}(e_1),\quad \sigma_2=\mrm{span}(e_2),\quad \sigma_3=\mrm{span}(e_3)=\mrm{span}(-e_1-e_2).\]
        The toric variety $X_{U_{2,3}}$ is $\mbb{P}^2\setminus\{\text{3 distinct points}\}$.
    \end{egg}

    Let $\mcal{H}=\{H_i\}_{i\in E}$ be a hyperplane arrangement in a complex vector space $L$ and $M_{\mcal{H}}$ be the corresponding matroid. There is an embedding $\iota\colon W_{\mcal{H}}\hookrightarrow X_{M_\mcal{H}}$ constructed in the following manner. After choosing equations for the hyperplanes of $\mcal{H}$, one obtains a map $L\hookrightarrow \mbb{C}^E$ which descends to a map $\mbb{P}(L)\hookrightarrow \mbb{P}(\mbb{C}^E)$. This map has the property that the projectivization of every hyperplane of $\mcal{H}$ is the intersection of $\mbb{P}(L)$ with the corresponding coordinate hyperplane of $\mbb{P}(\mbb{C}^E)$. Let $X_{\mcal{A}_E}$ be the variety constructed from $\mbb{P}(\mbb{C}^E)$ by first blowing up all of the coordinate points, then blowing up all of the strict transforms of the coordinate lines, and so on. The wonderful variety $W_{\mcal{H}}$ is the strict transform of $\mbb{P}(L)$ under this sequence of blowups and there is a commutative diagram
    \[\begin{tikzcd}
    	{W_{\mcal{H}}} & {X_{\mcal{A}_E}} \\
    	{\mbb{P}(L)} & {\mbb{P}(\mbb{C}^E)}
    	\arrow[hook, from=1-1, to=1-2]
    	\arrow[from=1-1, to=2-1]
    	\arrow[from=1-2, to=2-2]
    	\arrow[hook, from=2-1, to=2-2]
    \end{tikzcd}\]
    where the vertical arrows are blow down maps.
    
    The variety $X_{\mcal{A}_E}$ is the toric variety of the Bergman fan of the Boolean matroid on the ground set $E$. The toric variety $X_{M_\mcal{H}}$ is the subvariety of $X_{\mcal{A}_E}$ consisting of all torus orbits of $X_{\mcal{A}_E}$ which meet $W_{\mcal{H}}$. Consequently, the embedding $W_\mcal{H}\hookrightarrow X_{\mcal{A}_E}$ factors through an embedding $\iota: W_{\mcal{H}} \hookrightarrow X_{M_{\mcal{H}}}$. 
    
    For us, the essential information about the embedding $\iota:W_{\mcal{H}}\hookrightarrow X_{M_{\mcal{H}}}$ is the following lemma. This lemma follows from the description of $W_{\mcal{H}}$ as the strict transform of $\mbb{P}(L)$ in $X_{\mcal{A}_E}$ and the fact that $\iota$ induces an isomorphism between the Chow ring of $X_{M_{\mcal{H}}}$ and $W_{\mcal{H}}$ \cite[Corollary 2]{FY04}. 
    \begin{lem}(cf.\cite[Corollary 2]{FY04})\label{lem:intersection}
        Let $\rho_G$ be a ray of the Bergman fan of $M_{\mcal{H}}$ corresponding to a proper, non-empty flat $G$,  $D_{\rho_G}$ be the corresponding toric divisor of $X_{M_{\mcal{H}}}$, and $D_G$ be the corresponding exceptional divisor of $W_{\mcal{H}}$. The smooth subschemes $W_\mcal{H}$ and $D_{\rho_G}$ meet properly in $X_{M_{\mcal{H}}}$ and their scheme-theoretic intersection is
        \[D_{\rho_G}\cap W_{\mcal{H}}=D_G. \]
    \end{lem}
  
\section{Recollection of formal group laws and algebraic cobordism}\label{section:formal-group-laws-and-algebraic-cobordism}
    In \cref{subsection:formal-group}, we recall the theory of formal group laws, and in Sections \ref{subsection:cobordism} and \ref{subsection:construction}, we recall the theory of algebraic cobordism and oriented cohomology theories. 
    
\subsection{Formal group laws}\label{subsection:formal-group}
    Formal group laws are an important technical tool in the study of oriented cohomology theories. We explain in \cref{subsection:cobordism} that every oriented cohomology theory comes equipped with a formal group law which governs how the first Chern class of the tensor product of two line bundles on a smooth complex variety can be expressed in terms of the first Chern classes of the individual line bundles. The axioms defining a formal group law reflect properties of the tensor product operation on line bundles, as we explain in \cref{rmk:FGL-motivation}. In this subsection, we review the basic properties of formal group laws that we need in this article, closely following the exposition in \cite{F68}.
    
	\begin{defn}\label{defn:FGL}(\cite[Ch. III, \S 1]{F68})
		Let $R$ be a commutative ring. A one-dimensional, commutative \textbf{formal group law} $(R,F)$ is a formal power series $F(x,y)=\sum_{i,j\geq 0}a_{i,j}x^iy^j\in R\llbracket x,y\rrbracket$, such that for all $w,v,z\in R\llbracket x,y\rrbracket$:
		\begin{enumerate}
			\item\label{itm:1} $F(0,v)=v$,
			\item\label{itm:2} $F(w,v)=F(v,w)$, and
			\item\label{itm:3} $F(w,F(v,z))=F(F(w,v),z)$.
		\end{enumerate}
       
	\end{defn}

    In this article, all formal group laws are one-dimensional and commutative. For brevity, we will drop the adjectives and refer to them simply as ``formal group laws."

    \begin{rmk}\label{rmk:FGL-motivation}
        The three axioms defining a formal group law in \cref{defn:FGL} reflect the fact that, for any three line bundles $\mcal{L},\mcal{M},\mcal{N}$ on a smooth, complex variety $X$:
        \begin{enumerate}
           \item $\mcal{O}_X\otimes_{\mcal{O}_X} \mcal{L}\simeq \mcal{L}$,
            \item $\mcal{L}\otimes_{\mcal{O}_X}  \mcal{M}\simeq \mcal{M}\otimes_{\mcal{O}_X}  \mcal{L}$, and
            \item $\mcal{L}\otimes_{\mcal{O}_X} (\mcal{M}\otimes_{\mcal{O}_X} \mcal{N})\simeq (\mcal{L}\otimes_{\mcal{O}_X} \mcal{M})\otimes_{\mcal{O}_X} \mcal{N} $.
        \end{enumerate}
     \end{rmk}

    \begin{rmk}\label{rem:FGL}
        Let $F(x,y)=\sum_{i,j\geq0}a_{i,j}x^iy^j\in R\llbracket x,y\rrbracket$ be a formal group law. Axiom (\ref{itm:1}) in \cref{defn:FGL} implies that $a_{0,0}=0$ and $a_{1,0}=a_{0,1}=1$. Axiom (\ref{itm:2}) implies $a_{i,j}=a_{j,i}$ for all $i,j$. The relations among the $a_{i,j}$ arising  from Axiom (\ref{itm:3}) are complicated.
    \end{rmk}
    
    \begin{lem}(\cite[Ch. I, Section 3, Proposition 1]{F68}) \label{lem:formal_inv}
		Given a formal group law $(R,F)$, there is a unique formal power series \[-_Fx:=-x-c_2x^2-c_3x^3-\dotsb\in R\llbracket x\rrbracket\] called the \textbf{formal inverse} of $(R,F)$ that satisfies $F(x,-_Fx)=0$.
	\end{lem}

    \begin{egg}
		The \textbf{additive formal group law} over $\mbb{Z}$ is $F_A(x,y)=x+y$, and $-_{F_A}x=-x$.
	\end{egg}
	\begin{egg}
		The \textbf{multiplicative-periodic formal group law} over $\mbb{Z}[\beta,\beta^{-1}]$ is $F_M(x,y)=x+y- \beta xy$, and $-_{F_M}x=\frac{-x}{1-\beta x}$. 
	\end{egg}

    Definition \ref{defn:Lazard} below is important for us, as the \textit{Lazard ring} is the algebraic cobordism ring of a point.  
    \begin{defn}\label{defn:Lazard}
		The \textbf{Lazard ring} $\mbb{L}^*$ is the quotient of the free $\mbb{Z}$-algebra with generators $a_{i,j}$, $i,j\geq0$, modulo the relations imposed by the axioms of the formal group law. The ring $\mbb{L}^*$ is graded, with the degree of the generator $a_{i,j}$ being $1-i-j$. The \textbf{universal formal group law} is the formal group law $(\mbb{L}^*,F_U)$ defined by $F_U(x,y)=\sum_{i,j\geq 0}a_{i,j}x^iy^j\in \mbb{L}^*\llbracket x,y\rrbracket$.
	\end{defn}

    The formal group law $(\mbb{L}^*, F_U)$ is universal in the following sense. This universal property follows immediately from the construction of  $(\mbb{L}^*, F_U)$.
    
    \begin{prop}\label{rmk:Lazard-ring} 
        For every formal group law $(R,F)$, there exists a unique ring homomorphism 
        \[f:\mbb{L}^*\to R\]
        such that 
        \[f_*F_U \coloneqq \sum_{i,j\geq 0}f(a_{i,j})x^iy^j = F. \]
	\end{prop}  

    A fundamental result of Lazard gives a particularly simple presentation of $\mbb{L}^*$. This presentation is not at all apparent from our definition of $\mbb{L}^*$.

    \begin{theo}[\cite{L55}]\label{theo:Lazard-ring}
        The Lazard ring is isomorphic to a polynomial ring in infinitely many variables $\mbb{L}^*\simeq \mbb{Z}[t_1,t_2,\ldots]$ where the generator $t_{i}$ has degree $-i$.
    \end{theo}
  
    For a formal group law $(R,F)$, elements $g_1,\dotsc,g_k\in R\llbracket x_1,\dotsc,x_n\rrbracket$, and $a\in \mbb{Z}$, we define
    \[g_1+_F g_2:= F(g_1,g_2);\quad\quad a\cdot_F g_1:=\underbrace{g_1+_F g_1+_F\dotsb+_F g_1}_{\text{$a$ times}};\quad\quad \left.\sum_{i\in[k]}\right.^F g_i :=g_1+_F g_2+_F\dotsb+_F g_k.\]
    
    We now record two lemmas, \cref{lem:distributive-FGL} and \cref{lem:formal_dif}, for later use. \cref{lem:distributive-FGL} follows immediately from the definitions. As we are not aware of a reference for \cref{lem:formal_dif}, we include its proof.
    \begin{lem}\label{lem:distributive-FGL}
        Let $(R,F)$ be a formal group law. For any $a,b\in\mbb{Z}$ and $g\in R\llbracket x\rrbracket$, we have
        \[(a+b)\cdot_F g=(a\cdot_F g)+_F(b\cdot_F g).\]
    \end{lem}

    \begin{lem}
    \label{lem:formal_dif}
        Let $(R,F)$ be a formal group law and suppose $g_1,\ldots , g_k \in R\llbracket x_1,\dotsc,x_n\rrbracket$, where each $g_i$ has constant term $0$. If $I\subseteq[k]$, then there are $h_j\in R\llbracket x_1,\dotsc,x_n\rrbracket$, with $j\in [k]\setminus I$ and each $h_j$ having constant term $1$, such that 
        \[\left(\left.\sum_{i\in [k]}\right.^F g_i\right) - \left(\left.\sum_{i\in I}\right.^F g_i\right)=\sum_{j\in [k]\setminus I}g_j h_j. \]
    \end{lem}
    \begin{proof}
        Without loss of generality, we assume that $I=[\ell]$ for some $\ell< k$. By definition, we have
        \begin{align*}
        \left.\sum_{i\in [k]}\right.^Fg_i&=(g_1+_F\dotsb+_Fg_\ell)+_F(g_{\ell+1}+_F\dotsb+_F g_k)\\
        &=(g_1+_F\dotsb+_Fg_\ell)+(g_{\ell+1}+_F\dotsb+_F g_k)+\sum_{i,j\geq 1}a_{i,j}(g_1+_F\dotsb+_Fg_\ell)^i(g_{\ell+1}+_F\dotsb+_F g_k)^j.
        \end{align*}
        Thus, 
        \begin{align*}
        \left(\left.\sum_{i\in [k]}\right.^F g_i\right) - \left(\left.\sum_{i\in I}\right.^F g_i\right)&=(g_{\ell+1}+_F\dotsb+_Fg_k)+\sum_{i,j\geq 1}a_{i,j}(g_1+_F\dotsb+_Fg_\ell)^i(g_{\ell+1}+_F\dotsb+_F g_k)^j\\
        &=(g_{\ell+1}+_F\dotsb+_F g_k)\left(1+\sum_{i,j\geq 1}a_{i,j}(g_1+_F\dotsb+_Fg_\ell)^{i}(g_{\ell+1}+_F\dotsb+_F g_k)^{j-1}\right).     
        \end{align*}
        
        If $\ell=k-1$, then we can take $h_{k}:=\left(1+\sum_{i,j\geq 1}a_{i,j}(g_1+_F\dotsb+_Fg_\ell)^{i}(g_{\ell+1}+_F\dotsb+_F g_k)^{j-1}\right)$. As each $g_i$ has constant term $0$, the constant term of $h_k$ is $1$, concluding the proof of the lemma in this case.
        
        Now assume that $\ell<k-1$. We claim that $(g_{\ell+1}+_F\dotsb+_F g_k)=g_{\ell+1}h_{\ell+1}'+\dotsb+ g_kh_k'$, where each $h_j'\in R\llbracket x_1,\dotsc,x_n\rrbracket$ has constant term $1$. If this claim is true, then setting
        \[h_j:=h_j'\cdot\left(1+\sum_{i,j\geq 1}a_{i,j}(g_1+_F\dotsb+_Fg_\ell)^{i}(g_{\ell+1}+_F\dotsb+_F g_k)^{j-1}\right)\]
        concludes the proof of the lemma. 
        
        We will now prove the claim, using induction on the number of summands $s$ that appear in $g_1+_F\dotsb 
        +_F g_s$. For the base case, we note that $g_1=g_1\cdot h_1'$, where $h_1'=1$. Now assume that $g_1+_Fg_2+_F\dotsb+_F g_{s-1}=g_1h_1'+g_2h_2'+\dotsb+ g_{s-1}h_{s-1}'$ for some $h_{1}',h_{2}',\dotsc,h_{s-1}'\in R\llbracket x_1,\dotsc,x_n\rrbracket$, where each $h_{i}'$ has constant term $1$. We have
        \begin{align*}
        (g_1+_F\dotsb +_F g_{s})&=(g_1+_F\dotsb+_F g_{s-1})+_Fg_s\\
        &=(g_1+_F\dotsb+_F g_{s-1})+g_s+\sum_{i,j\geq 1}a_{i,j}(g_1+_F\dotsb+_F g_{s-1})^i g_{s}^j  \\
        &=g_1h_1'+\dotsb+g_{s-1}h_{s-1}'+g_s\left(1+\sum_{i,j\geq 1}a_{i,j}(g_1+_F\dotsb+_F g_{s-1})^ig_s^{j-1}\right).
        \end{align*}
        Setting $h_{s}':=\left(1+\sum_{i,j\geq 1}a_{i,j}(g_1+_F\dotsb+_F g_{s-1})^ig_s^{j-1}\right)$ concludes the induction step.
    \end{proof}

\subsection{Introduction to algebraic cobordism}\label{subsection:cobordism}
    Algebraic cobordism was constructed by Levine and Morel in \cite{LM07}. In this subsection, we will give a very brief summary of the construction of the algebraic cobordism ring from \cite{LM07}. We will outline the construction of algebraic cobordism in detail in \cref{subsection:construction}.

    Let $\textbf{Sch}_{\mbb{C}}$ be the category of separated schemes of finite type over $\mbb{C}$, and let $\textbf{Sm}_{\mbb{C}}$ be the full subcategory of $\textbf{Sch}_{\mbb{C}}$ consisting of schemes that are smooth and quasi-projective over $\mbb{C}$. Denote the category of commutative, $\mbb{Z}$-graded rings by $\mathbf{Ring}$. An oriented cohomology theory is a contravariant functor $\mrm{h}^*\colon \textbf{Sm}_{\mbb{C}}\to \mathbf{Ring}$\footnote{If $\mrm{h}^*$ is an oriented cohomology theory and $X\in\textbf{Sm}_{\mbb{C}}$, then we emphasize that the definition of an oriented cohomology theory (\`a la \cite{LM07}) requires that the ring $\mrm{h}^*(X)$ is commutative and graded, and does \textit{not} require that $\mrm{h}^*(X)$ is graded commutative. This choice of convention is explained in \cite[page VIII and page 2]{LM07}.}, satisfying several axioms that generalize the usual Eilenberg--Steenrod axioms from algebraic topology\footnote{In \cite{LM07}, Levine--Morel defined \emph{oriented cohomology theories} for schemes over an arbitrary base field. For simplicity, we focus our attention to schemes defined over the base field $\mbb{C}$ in this paper.}. A morphism of oriented cohomology theories is a natural transformation of functors that commutes with push-forwards. Examples of oriented cohomology theories include the \textit{Chow theory} $\CH^*$ that sends a scheme $X$ to its Chow ring $\CH^*(X)$ (\cref{egg:CH}), the \textit{graded $K$-theory} $K^*$ that sends a scheme $X$ to $K^0(X)\otimes_{\mbb{Z}} \mbb{Z}[\beta,\beta^{-1}]$ where $K^0(X)$ is the Grothendieck ring of algebraic vector bundles on $X$ and $\beta$ is a formal variable of degree $-1$ that artificially introduces a grading (\cref{egg:graded-K}), and the universal oriented cohomology theory \textit{algebraic cobordism} $\Omega^*$, which we now briefly describe.
        
    For a scheme $X\in \textbf{Sm}_\mbb{C}$, the algebraic cobordism ring $\Omega^*(X)$ is a ring that is generated by projective morphisms $f:Y\to X$ from smooth, quasi-projective, and irreducible varieties $Y$, modulo certain relations. The group $\Omega^k(X)$ is generated by the classes $[f\colon Y\to X]$ such that $k=\mrm{dim}(X)-\mrm{dim}(Y)$. The relations defining $\Omega^*(X)$, which we will recall in \cref{subsection:construction}, will imply that $\Omega^k(X)=0$ for all $k> \mrm{dim}(X)$. The algebraic cobordism ring $\Omega^*(\mrm{pt})$ of a point $\mrm{pt}=\mrm{Spec}(\mbb{C})$ is isomorphic to the Lazard ring $\mbb{L}^*$ and is generated by the classes of projective spaces $[\mbb{P}^n\to\mrm{pt}]$ and the classes of so-called Milnor hypersurfaces $[H_{n,m}\to\mrm{pt}]$ (\cite[Section 2.5.3]{LM07}). The structure map $X\to\mrm{pt}$ induces a pullback $\Omega^*(\mrm{pt})\to \Omega^*(X)$, giving $\Omega^*(X)$ the structure of a graded $\mbb{L}^*$-module.

    Given an oriented cohomology theory $\mrm{h}^*$ and a vector bundle $E\to X$ of rank $n$ in $\textbf{Sm}_{\mbb{C}}$, there are unique elements $c_i(E)\in \mrm{h}^*(X)$, for $i=0,\dotsc,n$, characterized by the axioms in \cite[pp. 3]{LM07}, which include the Whitney sum formula and the requirement that the $c_i$ commute with pullbacks. The elements $c_i(E)$ are called \textbf{Chern classes}. 
        
    Every oriented cohomology theory $\mrm{h}^*$ gives rise  to a formal group law $(\mrm{h}^*(\mrm{pt}), F)$. The power series $F$ is uniquely determined by the \textbf{Quillen formula}: if $\mcal{L}_1\to X$ and $\mcal{L}_2\to X$ are line bundles in $\textbf{Sm}_{\mbb{C}}$, then $c_1(\mcal{L}_1\otimes \mcal{L}_2)=F(c_1(\mcal{L}_1),c_1(\mcal{L}_2))$.
	
	\begin{egg}(\cite[Example 1.1.4]{LM07})\label{egg:CH}
        The \textbf{Chow theory} $\mrm{CH}^*$ which sends $X\in\textbf{Sm}_{\mbb{C}}$ to its Chow ring $\mrm{CH}^*(X)$ of algebraic cycles modulo rational equivalence defines an oriented cohomology theory. 
		The formal group law associated to $\CH^*$ is the {additive formal group law} over $\mrm{CH}^*(\mrm{pt})\simeq \mbb{Z}$. Indeed, for any $X\in \textbf{Sm}_{\mbb{C}}$ and line bundles $\mcal{L}_1$, $\mcal{L}_2$ on $X$, we have
        $c_1(\mcal{L}_1\otimes \mcal{L}_2)=c_1(\mcal{L}_1)+c_1(\mcal{L}_2)$.
    \end{egg}

    \begin{egg}(\cite[Example 1.1.5]{LM07})\label{egg:graded-K}
        Given $X\in\textbf{Sm}_{\mbb{C}}$, let $K^0(X)$ be the Grothendieck ring of algebraic vector bundles on $X$. The \textbf{graded $K$-theory} is the functor $K^*$ that sends $X\in\textbf{Sm}_{\mbb{C}}$ to the graded ring $K^0(X)\otimes_{\mbb{Z}}\mbb{Z}[\beta,\beta^{-1}]$, where $\mrm{deg}(\beta)=-1$. Given any morphism $f\colon X\to Y$ in $\textbf{Sm}_{\mbb{C}}$, the pullback in $K^*$ is defined by the formula
        \[f^*([\mcal{E}]\cdot \beta^n):=[f^*(\mcal{E})]\cdot \beta^{n}\in K^0(X)\otimes_{\mbb{Z}}\mbb{Z}[\beta,\beta^{-1}],\]
        where $\mcal{E}$ is a vector bundle on $Y$ and $n\in\mbb{Z}$. Given any projective morphism $f\colon X\to Y$ of pure codimension $d$, the pushforward in $K^*$ is defined by
        \[f_*([\mcal{E}]\beta^{n}):=\sum_{i=0}^\infty(-1)^i[R^i f_*(\mcal{E})] \cdot \beta^{n-d}\in K^0(Y)\otimes_{\mbb{Z}}\mbb{Z}[\beta,\beta^{-1}].\]
        Given a line bundle $\mcal{L}$ on $X\in\textbf{Sm}_{\mbb{C}}$, the first Chern class of $\mcal{L}$ is $c_1(\mcal{L})=(1-[\mcal{L}^\vee])\cdot \beta^{-1}$.
        The formal group law corresponding to $K^*$ is the multiplicative-periodic formal group law over $K^*(\mrm{pt})\simeq \mbb{Z}[\beta,\beta^{-1}]$. Indeed, given any two line bundles $\mcal{L}_1$ and $\mcal{L}_2$ on $X\in\textbf{Sm}_{\mbb{C}}$, we have $c_1(\mcal{L}_1\otimes \mcal{L}_2)=F_M(c_1(\mcal{L}_1),c_1(\mcal{L}_2))$.
    \end{egg}

    \begin{egg}\label{egg-complex-cobordism}
        Consider the complex cobordism functor $MU^*$, interpreted by Quillen as a complex-oriented cohomology theory \cite{Q69}. The functor $MU^*$ takes an $n$-dimensional complex manifold $X$ to its (graded) complex cobordism ring $MU^*(X)$. See, e.g., \cite[Section 1]{T97} for exposition on complex cobordism. The degree $2n-i$ piece, $MU^{2n-i}(X)$, is the free abelian group generated by the proper maps $M\to X$, where $M$ is a stably almost complex manifold  of real dimension $i$, modulo certain relations. The functor $MU^*$ is \textit{universal} among complex-oriented cohomology theories, in the sense that, for any complex-oriented cohomology theory $\mrm{h}^*$, there is a unique morphism of complex-oriented cohomology theories $MU^*\to \mrm{h}^*$ \cite{Q71}. As with algebraic cobordism, the complex cobordism ring of a point is isomorphic to the Lazard ring $\mbb{L}^*$.

        Let $X\in\textbf{Sm}_{\mbb{C}}$. The set of complex points $X(\mbb{C})$ of $X$ is a complex manifold. The functor $MU^*$ on $\textbf{Sm}_{\mbb{C}}$ which takes a scheme $X$ to the complex cobordism ring $MU^{2*}(X(\mbb{C}))$ is an oriented cohomology theory. The formal group law arising from $MU^*$ is the universal formal group law $F_U$ over $\mbb{L}^*$.
    \end{egg}

    \begin{rmk}\label{rmk:cobordismComparison}
        The oriented cohomology theories $MU^*$ and $\Omega^*$ have the same formal group law. However, $\Omega^*(-)\not\simeq MU^*(-)$ as oriented cohomology theories. For example, if $X$ is an elliptic curve, then $X(\mbb{C})$ is homeomorphic to a real torus and
        \[\Omega^*(X)\otimes_{\mbb{L}^*} \mbb{Z} \simeq \CH^*(X) \not\simeq H^*(X(\mbb{C})) \simeq MU^*(X(\mbb{C})) \otimes_{\mbb{L}^*} \mbb{Z} \]
        as abelian groups. This tells us that the formal group law does \textit{not} uniquely identify an oriented cohomology theory. 
    \end{rmk}

    Algebraic cobordism is the universal oriented cohomology theory in the following sense.
    \begin{theo}\label{theo:universal}(\cite[Theorem 1.2.6]{LM07})
        Given an oriented cohomology theory $\mrm{h}^*$, there is a unique morphism of oriented cohomology theories $\vartheta\colon \Omega^*\to \mrm{h}^*$.\footnote{In fact, Levine-Morel proved that algebraic cobordism is the universal oriented cohomology theory for schemes defined over \textit{any} base field of characteristic $0$. 
        }
    \end{theo}
    
    The following comparison theorems of Levine-Morel, \cref{theo:universal-CH} and \cref{theo:universal-K}, tell us how the Chow theory $\mrm{CH}^*$ and graded $K$-theory $K^*$ can be recovered from algebraic cobordism $\Omega^*$ through the universal morphisms $\Omega^*\to\mrm{CH}^*$ and $\Omega^*\to K^* $ established by \cref{theo:universal}. These comparison theorems show that $\mrm{CH}^*$ and $K^*$ lie in a family of oriented cohomology theories called \textit{free oriented cohomology theories}. We define free oriented cohomology theories in \cref{dfn:free-theory} and record the fact that $\mrm{CH}^*$ and $K^*$ are free oriented cohomology theories in \cref{cor:free-OCT}. 
    
    \begin{theo}\label{theo:universal-CH}(\cite[Theorem 1.2.19]{LM07})
		The universal morphism $\Omega^*\to \CH^*$ induces an isomorphism of oriented cohomology theories $\Omega^*(-)\otimes_{\mbb{L}^*}\mbb{Z}\xrightarrow{\sim}\CH^*(-)$, where the map $\mbb{L}^*\to \mbb{Z}$ is induced by the additive group law $F_A(x,y)=x+y$ over $\mbb{Z}$ (see \cref{rmk:Lazard-ring}).
	\end{theo}

    \begin{theo}\label{theo:universal-K}(\cite[Theorem 1.2.18]{LM07})
		The universal morphism $\Omega^*\to K^*$ induces an isomorphism of oriented cohomology theories
			$\Omega^*(-)\otimes_{\mbb{L}^*}\mbb{Z}[\beta,\beta^{-1}]\xrightarrow{\sim} K^*(-)$, where the map $\mbb{L}^*\to \mbb{Z}[\beta,\beta^{-1}]$ is      induced by the multiplicative group law $F_M(x,y)=x+y-\beta xy$ over $\mbb{Z}[\beta,\beta^{-1}]$ (see \cref{rmk:Lazard-ring}).
	\end{theo}

    \begin{rmk}\label{rmk:k0fgl}
        For $X\in\textbf{Sm}_{\mbb{C}}$, one can recover the $K$-ring of $X$ by the ring isomorphism
        $K^*(X)/ \langle \beta -1 \rangle \simeq K^0(X)$. This ensures that, under the ring homomorphism $\mbb{L}^*\to \mbb{Z}$ induced by the formal group law $(\mbb{Z}, x+y-xy)$,  we have an isomorphism of $\mbb{Z}$-algebras
        \[\Omega^*(X) \otimes_{\mbb{L}^*} \mbb{Z} \xrightarrow{\sim} K^0(X).\]
    \end{rmk}

    We now explain in \cref{dfn:free-theory} what it means for an oriented cohomology theory to be \textit{free}, and in \cref{cor:free-OCT}, we record that both $\mrm{CH}^*$ and $K^*$ are free.
    \begin{defn}\label{dfn:free-theory}    
       Let $(R,F)$ be a formal group law. Any oriented cohomology theory of the form $\Omega^*(-)\otimes_{\mbb{L}^*}R$ is called a \textbf{free oriented cohomology theory}, where $\mbb{L}^*\to R$ is the ring homomorphism of \cref{rmk:Lazard-ring} induced by $(R,F)$. 
    \end{defn}
    
    \begin{cor}\label{cor:free-OCT}
         Both $\CH^*$ and $K^*$ are free oriented cohomology theories.
    \end{cor}
    As we explain in \cref{prop:exceptional-explanation} and \cref{cor:exceptional-explanation},
    the observation that both $\mrm{CH}^*$ and $K^*$ are free is key to showing that our result, \cref{thrm:simplicial}, can be specialized to recover the integral ring isomorphism of \cite{LLPP24} between the Chow ring and $K$-ring of a matroid. 

\subsection{Construction of algebraic cobordism}\label{subsection:construction}
    In this subsection, we will describe the construction of \textit{algebraic cobordism}. Our exposition will closely follow \cite[Section 2.1]{K12} and \cite[Chapter 2.1.2]{LM07}. 
    
    Let $X\in\textbf{Sm}_{\mbb{C}}$. Let $\mcal{Z}^*(X)$ be the free abelian group generated by \textbf{cobordism cycles}, which are tuples $\alpha=[f\colon Y\to X,L_1,\dotsc,L_r]$, where $f$ is a projective morphism, $Y\in\textbf{Sm}_{\mbb{C}}$, and the $L_i$ are line bundles on $Y$. As in \cite[Def. 2.1.6]{LM07}, we consider two cobordism cycles $[f\colon Y\to X,L_1,\dotsc,L_r]$ and $[f'\colon Y'\to X,L_1',\dotsc,L_r']$ to be equal if there is a triple $\Phi=(\phi\colon Y\to Y',\sigma,(\psi_1,\dotsc,\psi_r))$ consisting of (i) an isomorphism $\phi\colon Y\xrightarrow{\simeq} Y'$ of $X$-schemes; (ii) a bijection $\sigma\colon\{1,\dotsc,r\}\xrightarrow{\simeq} \{1,\dotsc,r'\}$ (so that $r$ must equal $r'$); and (iii) for each $i=1,\dotsc,r$, an isomorphism $\psi_i\colon L_i\xrightarrow{\simeq} \phi^*(L_{\sigma(i)}')$ of line bundles on $Y$. The \textbf{degree} $\mrm{deg}(\alpha)$ of the cobordism cycle $\alpha=[f\colon Y\to X,L_1,\dotsc,L_r]$ is $\mrm{dim}_{\mbb{C}}(Y)-r$. The group $\mcal{Z}^*(X)$ is graded by the codimension of a cobordism cycle,  where the \textbf{codimension} of $\alpha$ is defined to be $\mrm{dim}_{\mbb{C}}(X)-\mrm{deg}(\alpha)$. That is, $\mcal{Z}^j(X)$ is generated as an abelian group by the cycles $[f\colon Y\to X,L_1,\dotsc,L_r]$ with $j=\mrm{dim}_{\mbb{C}}(X)-\mrm{dim}_{\mbb{C}}(Y)+r$. We will now impose several relations on $\mcal{Z}^*(X)$ to define the algebraic cobordism \textit{group} $\Omega^*(X)$. We will discuss the ring structure on $\Omega^*(X)$ later in this subsection.
    
	\begin{enumerate}
		\item\label{itm:dimension} (\textit{Dimension Axiom})  Let $\mcal{R}^*_{\mrm{dim}}(X)$ be the graded subgroup of $\mcal{Z}^*(X)$ generated by the cobordism cycles $[f\colon Y\to X,L_1,\dotsc,L_r]$ with $\mrm{dim}_{\mbb{C}}(Y)<r$. Define \[\mcal{Z}_{\mrm{dim}}^*(X):=\frac{\mcal{Z}^*(X)}{\mcal{R}^*_{\mrm{dim}}(X)}.\] Given any line bundle $L$ on $X$ and cobordism cycle $[f\colon Y\to X,L_1,\dotsc,L_r]$, one defines a \textbf{Chern class operator} $\widetilde{c_1}(L)\colon \mcal{Z}_{\mrm{dim}}^*(X)\to\mcal{Z}_{\mrm{dim}}^{*+1}(X)$ by \[\widetilde{c_1}(L)[f\colon Y\to X,L_1,\dotsc,L_r]:=[f\colon Y\to X,L_1,\dotsc,L_r,f^*(L)].\]
		\item\label{itm:section} (\textit{Section Axiom}) Let $\mcal{R}^*_{\mrm{sec}}(X)$ be the graded subgroup of $\mcal{Z}_{\mrm{dim}}^*(X)$	generated by the cobordism cycles of the form $[Y\to X,L]-[Z\to X]$, where $s\colon Y\to L$ is a section of the line bundle $L$ transverse to the zero-section, and $Z\hookrightarrow Y$ is the closed subvariety of $Y$ defined by the zeros of $s$. Define \[\underline{\Omega}^*(X):=\frac{\mcal{Z}^*_{\mrm{dim}}(X)}{\mcal{R}^*_{\mrm{sec}}(X)}.\]
		\item\label{itm:FGL} (\textit{Formal Group Law Axiom}) Let $\mcal{R}_{\mrm{FGL}}^*(X)$ be the graded $\mbb{L}^*$-submodule of $\mbb{L}^*\otimes_\mbb{Z}\underline{\Omega}^*(X)$ generated by elements of the form
		\[F_U(\widetilde{c_1}(L_1),\widetilde{c_1}(L_2))(\alpha)-\widetilde{c_1}(L_1\otimes L_2)(\alpha),\]
		over all $\alpha\in\underline{\Omega}^*(X)$, and all line bundles $L_1$, $L_2$ on $X$. The Dimension Axiom (\ref{itm:dimension}) implies that $\widetilde{c_1}(L)^{\mrm{dim}_{\mbb{C}}(Y)-r+1}([f\colon Y\to X,L_1,\dotsc,L_r])=0$, when $\mrm{dim}_{\mbb{C}}(Y)>r$. This means that all but finitely many terms of $F_U(\widetilde{c_1}(L_1),\widetilde{c_1}(L_2))(\alpha)$ are zero.
	\end{enumerate}
    
    \begin{defn}
		The \textbf{algebraic cobordism group} of $X$ is 
		\[\Omega^*(X):=\frac{\mbb{L}^*\otimes_\mbb{Z}\underline{\Omega}^*(X)}{\mcal{R}^*_{\mrm{FGL}}(X)}.\]
	\end{defn}
    
    \Cref{lem:simple-generation} below gives a simple generating set for $\Omega^*(X)$.
	\begin{lem}\label{lem:simple-generation}(\cite[Lemma 2.5.11]{LM07})
		Let $X\in\textbf{Sm}_{\mbb{C}}$. The graded, abelian group $\Omega^*(X)$ is generated as a group by the classes $[Y\to X]$ of projective morphisms $Y\to X$ with $Y$ smooth, quasi-projective, and irreducible.
	\end{lem}

    If $g\colon X\to X'$ is a projective morphism in $\textbf{Sm}_{\mbb{C}}$ of relative codimension $d$, then the \textbf{pushforward along $g$} is the group homomorphism defined by post-composition:
    \[g_*\colon \Omega^*(X)\to \Omega^{*+d}(X'),\]
    \[g_*([f\colon Y\to X,L_1,\dotsc,L_r]):=[g\circ f\colon Y\to X',L_1,\dotsc,L_r].\]
    
    If $g\colon X\to X'$ is an arbitrary morphism in $\textbf{Sm}_\mbb{C}$, then $g$ defines a \textbf{pullback map} $g^*\colon \Omega^*(X')\to \Omega^*(X)$ which is difficult to describe in general; see \cite[Section 6.5, Section 6.6]{LM07}. The main difficulty is that if $[Y\to X']$ is a class in $\Omega^*(X')$, then one would expect $g^*([Y\to X'])$ to be represented by the map $Y \times_{X'} X \to X$. However, $Y \times_{X'} X$ often fails to be smooth and thus the map $Y\times_{X'} X \to X$ does not define a class in $\Omega^*(X)$. However, in special situations, this is exactly how pullbacks act.

    \begin{lem}(\cite[Corollary 6.5.5]{LM07})\label{lem:smooth_pullbacks}
        Let $Z,X \in \textbf{Sm}_{\mbb{C}}$, $i:Z \to X$ be a regular embedding, and $[Y\to X]\in \Omega^*(X)$. If $Y\times_X Z$ is in $\textbf{Sm}_{\mbb{C}}$, then $i^*([Y\to X])= [Y\times_X Z \to Z]$.
    \end{lem}
    
    To construct a cup product for $\Omega^*(X)$, we first define an \textbf{external product} $\times: \Omega^*(X) \otimes \Omega^*(Y) \to \Omega^*(X\times Y)$ for $X$ and $Y$ in $\textbf{Sm}_\mbb{C}$. The external product is a group homomorphism and is defined as 
    \begin{align*}
        [f&\colon X'\to X,L_1,\dotsc,L_r]\times [g\colon Y'\to Y,M_1,\dotsc,M_s]\\&=[f\times g\colon X'\times Y'\to X\times Y,p_{X'}^*(L_1),\dotsc,p_{X'}^*(L_r),p_{Y'}^*(M_1),\dotsc,p_{Y'}^*(M_s)],
     \end{align*}
     where $p_{X'}\colon X'\times Y'\to X'$ and $p_{Y'}\colon X'\times Y'\to Y'$ are the projections onto $X'$ and $Y'$, respectively. Now consider the diagonal embedding $\delta\colon X\to X\times X$ in $\textbf{Sm}_{\mbb{C}}$. For any $X\in\textbf{Sm}_{\mbb{C}}$, the \textbf{cup product} on $\Omega^*(X)$ is defined by \[\cup:=\delta^*\circ \times \colon \Omega^*(X)\times\Omega^*(X)\to \Omega^*(X\times_{\mbb{C}} X)\to \Omega^*(X).\]
     As the cup product is defined using $\delta^*$, it can often be hard to compute the product of two classes in $\Omega^*(X)$. One situation in which the product is easy to compute is the following. 

     \begin{lem}\label{lem:disjoint-cycles}
        Let $D, E, X\in \textbf{Sm}_{\mbb{C}}$, and suppose $f\colon D\hookrightarrow X$ and $g\colon E\hookrightarrow X$ are regular embeddings. Consider the respective cobordism classes $[f\colon D\to X]$ and $[g\colon E\to X]$ in $\Omega^*(X)$. Assume that $f(D)\cap g(E)=\emptyset$. Then $[f\colon D\to X]\cup [g\colon E\to X]=0$ in $\Omega^*(X)$.
        \end{lem}
    Lemma~\ref{lem:disjoint-cycles} follows from Lemma~\ref{lem:smooth_pullbacks}. We close this section by recalling the definition of the first Chern class of a line bundle in algebraic cobordism.
    
    \begin{defn}(\cite[Section 5.2]{LM07})
        The \textbf{first Chern class} in $\Omega^*$ of a line bundle $L\to X$ in $\textbf{Sm}_{\mbb{C}}$ is defined to be \[c_1(L):=\widetilde{c_1}(L)(1_X).\]
    \end{defn}
    
    Let $D$ be a smooth divisor in $X\in\textbf{Sm}_{\mbb{C}}$. The Section Axiom (\ref{itm:section}) implies that 
	\[c_1(\mcal{O}_X(D))=[D\to X]\in\Omega^*(X).\]
    The first Chern class $c_1$ satisfies the property that for any line bundle $L\to X$ in $\textbf{Sm}_{\mbb{C}}$, we have $c_1(L)=s^*s_*(1_X)\in\Omega^1(X)$, where $s\colon X\to L$ denotes the zero section.
    
\section{Algebraic cobordism ring of $X_M$}\label{section:algebraic-cobordism-XM}
    In \cref{subsection:toric}, we recall the presentation of the algebraic cobordism ring of a smooth toric variety from \cite{KU13}. In \cref{subsection:FY}, we give a Feichtner-Yuzvinsky style presentation for the algebraic cobordism ring of a matroid. In \cref{subsection:simplicial}, we give a simplicial presentation for the algebraic cobordism ring of a matroid. The main result of this section is \cref{thrm:simplicial}.
    
    Recall the notation $(\mbb{L}^*,F_U)$ for the universal formal group law over the Lazard ring $\mbb{L}^*$.

\subsection{Algebraic cobordism rings of toric varieties}\label{subsection:toric}
     In \cite[Theorem 1.2]{KU13}, Krishna and Uma give a presentation for the algebraic cobordism ring of a smooth toric variety. We recall their presentation in \cref{thrm:KU} and compare it to the standard presentation for the Chow ring of a smooth toric variety in \cref{rmk:compare-toric}.

     Let $\Sigma$ be a unimodular fan in $N_{\mathbb{R}}$ with $n$-dimensional lattice $N$ and  {smooth} toric variety $X_\Sigma$. For each ray $\rho$ in $\Sigma$, let $v_\rho$ be the primitive ray generator of $\rho$, and let $D_{\rho}$ be the divisor in $X_\Sigma$ corresponding to $\rho$. The inclusion $D_\rho\hookrightarrow X_\Sigma$ defines a class $[D_\rho\hookrightarrow X_\Sigma]$ in $\Omega^*(X_\Sigma)$. Consider the formal power series ring over $\mbb{L}^*$ with variables indexed by the rays in $\Sigma$  \[T_{\Sigma}:=\mbb{L}^*\llbracket  x_\rho\text{ $\colon$ } \text{$\rho$ is a ray in $\Sigma$} \rrbracket.\]
    Consider the following ideals in $T_\Sigma$:
    \begin{align*}
    I_1^\Sigma:=&\langle x_{\rho_1}\cdots x_{\rho_k}:\text{ the rays } \rho_1,\ldots,\rho_k \text{ do not generate a cone of }\Sigma \rangle,\\
    I_2^\Sigma:=&\left\langle \left.\sum_{\text{$\rho$ ray}}\right.^{F_U} \langle \chi, v_\rho \rangle \cdot_{F_U} x_\rho : \chi \in N^\vee\right \rangle,\text{ and}\\
    I^\Sigma_{\text{nil}}:=&\langle x_\rho^{n+1} : \text{$\rho$ is a ray in $\Sigma$}\rangle.
    \end{align*}

    In the proof of \cite[Theorem 8.2]{KU13}, it is shown that $I_{\text{nil}}^\Sigma \subseteq I_1^\Sigma +I_2^\Sigma$ and that the following isomorphisms hold. We choose to include the ideal $I_{\text{nil}}^\Sigma$ in our presentation because it will make later arguments more transparent.
    
    \begin{theo}(\cite[Theorem 1.2]{KU13}) \label{thrm:KU} 
    There is an $\mbb{L}^*$-algebra isomorphism
    \[\frac{T_\Sigma}{I_1^\Sigma+I_2^\Sigma + I^\Sigma_{\text{nil}}} \to \Omega^*(X_\Sigma),\quad x_\rho\mapsto [D_\rho\hookrightarrow X_\Sigma].\]
    \end{theo}
    
    \begin{rmk}\label{rmk:compare-toric}
        Using \cref{thrm:KU} and \cref{theo:universal-CH}, one recovers the standard presentation for 
        $\mrm{CH}^*(X_M)$ as $\mrm{CH}^*(X_M)\simeq \Omega^*(X_M)\otimes_{\mbb{L}^*}\mbb{Z}$. Heuristically, simply replace the pair $(\mbb{L}^*,F_U)$ with the pair $(\mbb{Z},F_A)$ everywhere in sight: $T_\Sigma$ becomes $\mbb{Z}[  x_\rho : \rho \text{ a ray of }\Sigma]$, $I_2^\Sigma$ is generated by the ordinary sums $\sum_{\text{$\rho$ ray}} \langle \chi, v_\rho \rangle x_\rho$, where $\chi \in N^\vee$, and $I_1^\Sigma$ remains the same.
    \end{rmk}

     Let $\chi_1,\dotsc,\chi_n$ be generators for the character lattice $N^\vee$. Consider the following ideal in $T_\Sigma$:
     \[(I_2^\Sigma)':=\left\langle \left.\sum_{\text{$\rho$ ray}}\right.^{F_U} \langle \chi_i, v_\rho \rangle \cdot_{F_U} x_\rho : i\in [n] \right\rangle.\]
     Observe that $(I_2^\Sigma)'$ is contained in $I_2^\Sigma$. We will now prove \cref{lem:J_basis}, which implies that only finitely many relations are required in the ring presentation for $\Omega^*(X_\Sigma)$ of \cref{thrm:KU}.    
      
    \begin{lem}\label{lem:J_basis}
    There is an $\mbb{L}^*$-algebra isomorphism
    \[\frac{T_\Sigma}{I_1^\Sigma+(I_2^\Sigma)'+I_{\text{nil}}^\Sigma}\to \Omega^*(X_\Sigma),\quad x_\rho\mapsto [D_\rho \hookrightarrow X_{\Sigma}].\]
    \end{lem}
    \begin{proof}
    In light of Theorem~\ref{thrm:KU} and the fact that $(I_2^\Sigma)' \subseteq I_2^\Sigma$, it suffices to prove that the image of $I_2^\Sigma$ in $T_{\Sigma}/(I_1^{\Sigma}+I_\text{nil}^\Sigma)$ is contained in the image of $(I_2^\Sigma)'$ in $T_{\Sigma}/(I_1^{\Sigma}+I_\text{nil}^\Sigma)$. Let $\chi = \sum_{i=1}^n a_n \chi_n$ be an arbitrary character in $N^\vee$. It follows from \cref{lem:distributive-FGL} and the axioms of a formal group law that
    \begin{align*}
    \left.\sum_{\text{$\rho$ ray}}\right.^{F_U} \left\langle \chi, v_\rho \right\rangle \cdot_{F_U} x_\rho &= \left.\sum_{\text{$\rho$ ray}}\right.^{F_U} \left\langle \left(\sum_{1\leq i \leq n} a_i\chi_i\right), v_\rho \right\rangle \cdot_{F_U} x_\rho\\
    &= \left.\sum_{\text{$\rho$ ray}}\right.^{F_U} \left.\sum_{1\leq i \leq n}\right.^{F_U} \left(a_i\langle \chi_i, v_\rho \rangle\right) \cdot_{F_U} x_\rho\\&
    = \left.\sum_{1\leq i \leq n}\right.^{F_U} \left( a_i \cdot_{F_U} \left.\sum_{\text{$\rho$ ray}}\right.^{F_U} \langle \chi_i, v_\rho \rangle \cdot_{F_U} x_\rho\right).
    \end{align*}
    As $x_\rho$ is nilpotent in $T_{\Sigma}/(I_1^{\Sigma}+I_\text{nil}^\Sigma)$, it follows that the last expression is contained in the image of $(I_2^\Sigma)'$ in $T_{\Sigma}/(I_1^{\Sigma}+I_\text{nil}^\Sigma)$.
    \end{proof}
    
\subsection{The Feichtner-Yuzvinsky presentation}\label{subsection:FY}
    We give a Feichtner--Yuzvinsky style presentation for the algebraic cobordism ring of the toric variety of the Bergman fan of a loopless matroid in \cref{prop:FY-presentation}. 

    Let $M$ be a loopless matroid of rank $r$ on ground set $E$ with flats $\mcal{F}$. Recall the notation $X_M:=X_{\Sigma_M}$ for the toric variety of the Bergman fan $\Sigma_M$ of $M$. The character lattice of $\Sigma_M$ is the lattice $N_E^\vee:=(\mathbb{Z}^E/\mathbb{Z}e_E)^*$ in $(\mathbb{R}^E/\mathbb{R}e_E)^*$, and the set $\mcal{B}_M:=\{e_i^*-e_j^*: i,j\in E\}$ contains a basis for the lattice $N_E^\vee$. Consider the formal power series ring
    \[T_M:=\mathbb{L}^*\llbracket x_G : \text{ $G$ a proper and non-empty flat of $M$}\rrbracket\]
    and the following two ideals in $T_M$:
    \[I_1^M:= \langle x_Gx_{G'}: \text{ $G$, $G'$ incomparable flats}\rangle, \text{ and}\]
    \[I_2^M:= \left\langle \left.\sum_{i_0\in G}\right.^{F_U} x_G +_{F_U} \left.\sum_{j_0\in G'}\right.^{F_U} \left(-_{F_U}x_{G'}\right) : i_0\in E \text{ and }j_0 \in E\right \rangle.\]
    \begin{defn}
        The \textbf{algebraic cobordism ring of $M$} is
        \[\Omega^*(M):=\frac{T_M}{I_1^M+I_2^M}.\]
    \end{defn}

    A priori, it is unclear whether $\Omega^*(M)$ is isomorphic to $\Omega^*(X_M)$.
    For the remainder of this section, we verify that indeed $\Omega^*(M)\simeq \Omega^*(X_M)$. Consider the ideal $I^M_{\text{nil}}:=\langle x_G^{r} : \text{$G$ is a proper, non-empty flat of $M$}\rangle$ in $T_M$. Below, we prove \cref{lem:nilpotents-contained-in-smaller-ideal}, which implies that the generators $x_G$ are nilpotent in $\Omega^*(M)$.
     
     \begin{lem}\label{lem:nilpotents-contained-in-smaller-ideal}
        The ideal $\langle x_G : G \text{ a proper, non-empty flat of } M\rangle^r$
        is contained in $I_1^M+I_2^M$. In particular, the ideal
        $I^M_{\text{nil}}$ is contained in $I_1^M+I_2^M$.
     \end{lem}
     \begin{proof}
     To prove the claim, we show that if $x_{G_1}^{d_1}x_{G_2}^{d_2}\cdots x_{G_k}^{d_k}$ is a monomial with $d_1+d_2+\ldots+d_k \geq r$, then $x_{G_1}^{d_1}x_{G_2}^{d_2}\cdots x_{G_k}^{d_k} =0$ in $\Omega^*(M)$. By the relations of $I_1^M$, if $G_1,\ldots , G_k$ is not a  chain of proper, non-empty flats, then $x_{G_1}^{d_1}x_{G_2}^{d_2}\cdots x_{G_k}^{d_k} =0$. Thus, in particular, our claim holds when $k\geq r$. We will now assume that $G_1 \subsetneq G_2 \subsetneq \ldots \subsetneq G_k$ is a chain of proper, non-empty flats and prove our claim by induction on $r-k$. The base case where $r-k=0$ is covered by our previous observation. For the inductive step, let $r-k\geq 1$ and assume that every monomial $x_{G_1'}^{d_1'}\cdots x_{G_{k'}'}^{d_{k'}'}$ supported on a chain of flats with $d_1'+\ldots +d_{k'}' \geq r$ and $k'>k$ is equal to zero in $\Omega^*(M)$. As $r-k\geq 1$, there is an index $\ell$ such that $d_\ell\geq 2$.
     Fix elements $i_0\in  G_{\ell+1}\setminus G_\ell$ and $j_0\in G_\ell\setminus G_{\ell-1}$ where we say $G_0=\emptyset$ and $G_{k+1}=E$. By the relations of $I_2^M$, we have that 
     \[x_{G_\ell} = \left(\left.\sum_{i_0\in H}\right.^{F_U} x_H\right) +_{F_U} \left(\left.\sum_{j_0\in H', H'\neq G_\ell}\right.^{F_U} (-_{F_U}x_{H'})\right).\]
     Using this identity and the relations of $I_1^M$, we get
     \begin{align}
             x_{G_1}^{d_1}\cdots x_{G_\ell}^{d_\ell}\cdots x_{G_k}^{d_k} &= x_{G_1}^{d_1}\cdots x_{G_\ell}^{d_\ell-1}\cdots x_{G_k}^{d_k}\left(\left(\left.\sum_{i_0\in H}\right.^{F_U} x_H \right)+_{F_U} \left(\left.\sum_{j_0\in H', H'\neq G_\ell}\right.^{F_U} (-_{F_U}x_{H'})\right) \right)\nonumber
             \\&=
             x_{G_1}^{d_1}\cdots x_{G_\ell}^{d_\ell-1}\cdots x_{G_k}^{d_k}\left(\left(\left.\sum_{\substack{G_{\ell}\cup i_0\subseteq  H}}\right.^{F_U} x_H \right)+_{F_U} \left(\left.\sum_{\substack{ G_{\ell-1}\cup j_0\subseteq H', G_\ell\neq H'}}\right.^{F_U} (-_{F_U}x_{H'})\right) \right)\nonumber
             \\&=
             x_{G_1}^{d_1}\cdots x_{G_\ell}^{d_\ell-1}\cdots x_{G_k}^{d_k}\left(\left.\sum_{\substack{ G_{\ell-1}\cup j_0\subseteq H'\\ G_{\ell}\cup i_0\not\subseteq H',\text{ }G_\ell\neq H'}}\right.^{F_U} (-_{F_U}x_{H'})\right)\nonumber
             \\
         &= x_{G_1}^{d_1}\cdots x_{G_\ell}^{d_\ell-1}\cdots x_{G_k}^{d_k} \left(\left.\sum_{\substack{G_{\ell-1}\cup j_0 \subseteq H' \subsetneq G_{\ell+1} \\ i_0\not\in H', H'\neq G_\ell}}\right.^{F_U} (-_{F_U}x_{H'})\right).\label{eq:the_end}
     \end{align}
    By Lemma~\ref{lem:formal_inv}, the quantity $-_F x_{H'}$ in the sum of expression (\ref{eq:the_end}) can be written as $x_{H'} f_{H'}$ where $f_{H'}$ is an element of $T_M$. By Lemma~\ref{lem:formal_dif},
    the large formal sum of expression (\ref{eq:the_end}) can be written, in a highly non-canonical way, as a finite sum of terms of the form $x_{H'}f_{H'}g_{H'}$ where each $g_{H'}$ is an element of $T_M$. This lets us write
    \[x_{G_1}^{d_1}\cdots x_{G_\ell}^{d_\ell}\cdots x_{G_k}^{d_k} = \sum_{\substack{G_{\ell-1}\cup j_0 \subseteq H' \subsetneq G_{\ell+1} \\ i_0\not\in H', H'\neq G_\ell}} x_{G_1}^{d_1}\cdots x_{G_\ell}^{d_\ell-1}\cdots x_{G_k}^{d_k}x_{H'} f_{H'}g_{H'}.\]
    Each monomial $x_{G_1}^{d_1}\cdots x_{G_\ell}^{d_\ell-1}\cdots x_{G_k}^{d_k} x_{H'}$ is zero in $\Omega^*(M)$ by our inductive hypothesis. As $x_{G_1}^{d_1}\cdots x_{G_\ell}^{d_\ell}\cdots x_{G_k}^{d_k}$ is a finite sum of elements that are zero in $\Omega^*(M)$, our claim now follows.
    \end{proof}

\begin{prop}\label{prop:FY-presentation}
    There is an $\mbb{L}^*$-algebra isomorphism
    \[\varphi\colon \Omega^*(M)\to \Omega^*(X_M), \quad x_{G}\mapsto [D_{\rho_G} \hookrightarrow X_M]  \text{ for all $G\in\mcal{F}\setminus \{E,\emptyset\}$}.\]
\end{prop}
\begin{proof}
  By \Cref{lem:J_basis}, it will suffice to show that the isomorphism $\varphi:T_M \to T_{\Sigma_M}$ defined by $ \varphi(x_{G})= x_{\rho_G}$ has the property that
  \[\varphi(I_1^M)+ \varphi(I_2^M) = I_1^{\Sigma_M}+(I_2^{\Sigma_M})'+I_{\text{nil}}^{\Sigma_M}. \]
  If $G$  and $G'$ are incomparable flats, then there is no chain of flats which contains both $G$ and $G'$, so $\varphi(x_Gx_{G'})=x_{\rho_G}x_{\rho_{G'}}\in I_1^{\Sigma_M}$. Similarly, if the rays $\rho_{G_1},\dotsc,\rho_{G_k}$ do \textit{not} generate a cone in $\Sigma_M$, then there are two indices $i,j$ for which $G_i$ and $G_j$ are incomparable flats. This shows that $\varphi(I_1^M)=I_1^{\Sigma_M}$. After choosing $\mcal{B}_M$ as our set of generators for the $\mbb{Z}$-module $N_E^\vee$ and a small algebraic modification using \cref{lem:distributive-FGL}, we see that the generators of $I_2^M$ are mapped to the generators of $(I_2^{\Sigma_M})'$ by $\varphi$. The rest of the claim now follows from \cref{lem:nilpotents-contained-in-smaller-ideal}.
\end{proof}

\begin{rmk}\label{rmk:motivation}
    Let $\mrm{h}^*$ be any oriented cohomology theory and $(\mrm{h}^*(\mrm{pt}), F_{\mrm{h}^*})$ be the corresponding formal group law. Let $S$ be the sub $\mrm{h}^*(\mrm{pt})$-algebra of $\mrm{h}^*(X_M)$ generated by the first Chern classes of the line bundles $\mathcal{O}_{X_M}(D_{\rho_G})$ corresponding to the toric divisors of $X_M$. As $\Omega^*(X_M)$ is generated by first Chern classes of line bundles and as morphisms of oriented cohomology theories preserve first Chern classes, the image of the universal morphism $\Omega^*(X_M) \to \mrm{h}^*(X_M)$ is equal to $S$. \Cref{prop:FY-presentation} tells us that $S$ can be written as a quotient of the ring 
    \[\frac{\mrm{h}^*(\mrm{pt})[x_G: G \text{ a proper non-empty flat of } M]}{\langle x_Gx_{G'}: G,G' \text{ incomparable}\rangle + \langle\left.\sum_{i_0\in G}\right.^{F_{\mrm{h}^*}} x_G +_{F_{\mrm{h}^*}} \left.\sum_{j_0\in G'}\right.^{F_{\mrm{h}^*}} \left(-_{F_{\mrm{h}^*}}x_{G'}\right) : i_0,j_0\in E \rangle}. \]
    This fact can be observed independently of \Cref{prop:FY-presentation}. The monomial relations hold because if $G$ and $G'$ are incomparable, then $D_{\rho_G}\cap D_{\rho_{G'}}=\emptyset$ and $c_1(\mathcal{O}_{X_M}(D_{\rho_G}))\cdot c_1(\mathcal{O}_{X_M}(D_{\rho_{G'}}))= 0$ in $S$ by \Cref{lem:disjoint-cycles}. The second set of relations follow from Quillen's formula for computing the Chern class of a tensor product of line bundles and the relations of $\mrm{Pic}(X_M)$. Explicitly, for all $i_0,j_0\in E$, we have that \[\bigotimes_{i_0\in G}\mcal{O}_{X_M}(D_{\rho_G})\otimes \bigotimes_{j_0\in G'}\mcal{O}_{X_M}(-D_{\rho_{G'}})=\mcal{O}_{X_M}\quad\text{in } \mrm{Pic}(X_M)\]
    and thus
    \begin{align*}
    0=c_1(\mcal{O}_{X_M})&=c_1\left(\bigotimes_{i_0\in G}\mcal{O}_{X_M}(D_{\rho_G})\otimes \bigotimes_{j_0\in G'}\mcal{O}_{X_M}(-D_{\rho_{G'}})\right)
    \\&=\left.\sum_{i_0\in G}\right.^{F_{\mrm{h}^*}} (c_1(\mcal{O}_{X_M}(D_{\rho_G})))+_{F_{\mrm{h}^*}} \left.\sum_{j_0\in G'}\right.^{F_{\mrm{h}^*}} \left(c_1(\mcal{O}_{X_M}(-D_{\rho_{G'}})\right)
    \end{align*}
    in $S$. If $\mrm{h}^*$ is not a free oriented cohomology theory (\Cref{dfn:free-theory}), it is possible to have $\mrm{h}^*(X_M) \not\simeq S$; see \Cref{rmk:binder}.
\end{rmk}

    We now define a different presentation for $\Omega^*(M)$ in \cref{cor:second-FY-presentation}, which includes a generator $x_E$ for the non-proper flat $E$. This presentation follows from \cref{prop:FY-presentation} and the formal group law manipulation described in \cref{lem:distributive-FGL}. First, consider the formal power series ring, \[\mcal{T}_M:= \mathbb{L}\llbracket x_G : \text{ $G$ a non-empty flat of $M$}\rrbracket.\]
    We define the following two ideals in $\mcal{T}_M$:
    \[\mcal{I}_1^M:=\langle x_Gx_{G'}: \text{ $G$ and $G'$ are incomparable}\rangle,\quad\text{and}\quad \quad \mcal{I}_2^M:=\left\langle \left.\sum_{j_0\in G}\right.^{F_U} x_G : j_0\in E\right\rangle.\]

\begin{cor}\label{cor:second-FY-presentation}
    Fix any $i_0\in E$. There is an $\mbb{L}^*$-algebra isomorphism
    \[\frac{\mcal{T}_M}{\mcal{I}_1^M+\mcal{I}_2^M}\to  \frac{T_{\Sigma_M}}{I_1^{\Sigma_M}+I_2^{\Sigma_M}} = \Omega^*(X_M),\quad x_G\mapsto\begin{cases}
        x_G,\quad\quad \quad \quad \quad \quad  \quad \text{$G\neq E$};\\
        \left.\sum_{i_0\in H}\right.^{F_U} (-_{F_U}x_H),\quad \text{$G=E$},
    \end{cases}\]
    where in the case $G=E$, the sum is over \textit{all} flats $H$ that contain $i_0\in E$.
\end{cor}

\begin{rmk}\label{rmk:chow}
    The Feichtner--Yuzvinsky presentation for $\CH^*(X_M)$ from \cite{FY04} can be recovered from \cref{cor:second-FY-presentation} by \cref{theo:universal-CH}, and similarly the Feichtner--Yuzvinsky presentation for $K^0(X_M)$ from \cite{LLPP24} can be recovered from \cref{cor:second-FY-presentation} by \cref{theo:universal-K}. Heuristically, as in \cref{rmk:compare-toric}, in the presentation $\frac{\mcal{T}_M}{\mcal{I}_1^M+\mcal{I}_2^M}$, replace the pair $(\mbb{L}^*,F_U)$ with $(\mbb{Z},F_A)$ to obtain the Feichtner--Yuzvinsky presentation for $\CH^*(X_M)$, and replace $(\mbb{L}^*,F_U)$ with $(\mbb{Z},F_M)$ and set $\beta=1$ to obtain the Feichtner--Yuzvinsky presentation for $K^0(X_M)$.
\end{rmk}

\subsection{The simplicial presentation and its consequences}\label{subsection:simplicial}
    The goal of this subsection is to prove \cref{thrm:simplicial}, which is a \textit{simplicial presentation} for $\Omega^*(X_M)$ analogous to the simplicial presentations for $\CH^*(X_M)$ and $K^0(X_M)$. We also present \cref{cor:free} and \cref{cor:exceptional-explanation}, which explain some consequences of \cref{thrm:simplicial}. 

    We recall a basic fact about formal power series rings which we will use repeatedly in this section.
    \begin{lem}\label{lem:invertible}
        Let $R$ be any commutative ring, and set 
        \[\mcal{S}= R\llbracket x_1,\ldots,x_n\rrbracket. \]
        If $g\in \mcal{S}$ and the constant term of $g$ is a unit in $R$, then $g$ is a unit in $\mcal{S}$.
    \end{lem}

    Define
    \[S_M:= \mbb{Z}[h_G: \text{ $G$ a non-empty flat of $M$}]\quad \text{and}\quad \mcal{S}_M:= \mbb{L}^*\llbracket h_G: \text{ $G$ a non-empty flat of $M$}\rrbracket,\]
    and let
    \begin{gather*}
            J_1^M:=\langle (h_G-h_{G\vee H})(h_H-h_{G\vee H}): G,H \text{ non-empty flats of $M$}\rangle, \quad  J_2^M:=\langle h_G: \mrm{rank}_M(G)=1\rangle \subseteq S_M\\
      \mcal{J}_1^M:=\langle (h_G-h_{G\vee H})(h_H-h_{G\vee H}): G,H \text{ non-empty flats of $M$}\rangle,  \quad \mcal{J}_2^M:=\langle h_G: \mrm{rank}_M(G)=1\rangle \subseteq\mcal{S}_M
    \end{gather*}
    be ideals of $S_M$ and $\mcal{S}_M$, respectively.

    The ring presentation $\frac{S_M}{J_1^M+J_2^M}$ is equal to the \textbf{simplicial presentation} of the Chow ring of $X_M$, constructed in \cite{BES24} and explicitly described in \cite{LLPP24}. The elements in $\CH^*(X_M)$ that correspond to the generators $h_G$ are called \textbf{simplicial generators} of $\CH^*(X_M)$. In \cite{LLPP24}, the authors showed that $K^0(X_M)$ can be presented using the \textit{same} simplicial presentation. The elements in $K^0(X_M)$ corresponding to the $h_G$ are called the \textbf{simplicial generators} of $K^0(X_M)$. The ring isomorphism that sends a simplicial generator $h_G$ in $K^0(X_M)$ to the corresponding simplicial generator $h_G$ in $\CH^*(X_M)$ is referred to as the \textit{exceptional isomorphism} between $K^0(X_M)$ and $ \CH^*(X_M)$ in \cite{LLPP24}.
    
    \begin{lem}\label{lem:Ltensor}
        There is an $\mbb{L}^*$-algebra isomorphism
         \[\mbb{L}^*\otimes_{\mbb{Z}}\CH^*(X_M) \simeq \mbb{L}^* \otimes_{\mbb{Z}} \frac{S_M}{J_1^M+J_2^M} \to \frac{\mcal{S}_M}{\mcal{J}_1^M+\mcal{J}_2^M}, \quad \quad h_G \mapsto h_G. \]
    \end{lem}
    \begin{proof}
        The claim follows formally as soon we verify that the generators $h_G$ are nilpotent in $\frac{\mcal{S}_M}{\mcal{J}_1^M+\mcal{J}_2^M}$. Indeed, the elements 
        $h_G^r$ for non-empty flats $G$ are all contained in $\mcal{J}_1^M+\mcal{J}_2^M$. This can be shown directly from the definition of $\mcal{J}_1^M$ and $\mcal{J}_2^M$. Alternatively, this also follows from the fact that $\CH^*(X_M) \simeq \frac{S_M}{J_1^M+J_2^M}$ and  $\CH^k(X_M)=0$ for $k\geq r$; see \cite{LarsonStraightening}.
    \end{proof}
    We now prove a simplicial presentation for $\Omega^*(X_M)$.
    \begin{theo}\label{thrm:simplicial}
        There are $\mbb{L}^*$-algebra isomorphisms
        \[\mbb{L}^*\otimes_{\mbb{Z}} \CH^*(X_M) \simeq \frac{\mcal{S}_M}{\mcal{J}_1^M+\mcal{J}_2^M} \xrightarrow{\Phi}  \frac{\mcal{T}_M}{\mcal{I}_1^M+\mcal{I}_2^M}\simeq \Omega^*(X_M)\]
        where $\Phi$ is defined by declaring $\Phi(h_G)= \left.\sum_{H \supseteq G}\right.^{F_U} (-_{F_U}x_H)$.
    \end{theo}
\begin{proof}
By \Cref{cor:second-FY-presentation} and \Cref{lem:Ltensor}, all we need to show is that $\Phi$ is an isomorphism. Let $\mu$ be the 
M\"{o}bius function of the lattice of flats $M$. Consider the two $\mbb{L}^*$-algebra maps:
\[\phi: \mcal{S}_M \to T_M, h_G\mapsto \left.\sum_{G \subseteq H}\right.^{F_U} (-_{F_U}x_H) \quad \text{and} \quad \psi: T_M \to \mcal{S}_M, x_H\mapsto \left.\sum_{H\subseteq G}\right.^{F_U} \mu(H,G)\cdot_{F_U} (-_{F_U} h_G).\]
One can verify using M\"{o}bius inversion that $\phi$ and $\psi$ are inverse to each other. That is,
for a non-empty flat $H$ of $M$, we have 
\begin{align*}
\phi(\psi(x_H))&=
\phi\left( \left.\sum_{H\subseteq G}\right.^{F_U} \mu(H,G)\cdot_{F_U} (-_{F_U} h_G) \right) \\&=   \left.\sum_{H\subseteq G}\right.^{F_U} \left(\mu(H,G) \cdot_{F_U}  \left(\left.\sum_{G\subseteq K}\right.^{F_U} x_K\right)\right) \\&= \left.\sum_{H\subseteq K}\right.^{F_U}\left( \sum_{H\subseteq G\subseteq K} \mu(H,G) \right) \cdot_{F_U} x_K \\&= x_H 
\end{align*}
and
\begin{align*}
\psi(\phi(h_G))&=\psi\left( \left.\sum_{G\subseteq H}\right.^{F_U} (-_{F_U} x_H) \right) \\&=   \left.\sum_{G\subseteq H}\right.^{F_U} \left(-_{F_U}\left(  \left.\sum_{H\subseteq K}\right.^{F_U} \mu(H,K)\cdot_{F_U} (-_{F_U}h_K)\right)\right) \\&= \left.\sum_{G\subseteq H}\right.^{F_U}\left( \left.\sum_{H\subseteq  K}\right.^{F_U} \mu(H,K)  \cdot_{F_U} h_K\right) \\&=
\left.\sum_{G\subseteq K}\right.^{F_U}\left(\sum_{G\subseteq H\subseteq K}\mu(H,K)\right)\cdot_{F_U} h_K
\\&= h_G.
\end{align*}
Therefore, $\phi$ is in fact an isomorphism of $\mbb{L}^*$-algebras, with inverse $\psi$. 
To show that $\Phi$ is an isomorphism, it is enough to show that ${\phi}(\mcal{J}_1^M)=\mcal{I}_1^M$ and ${\phi}(\mcal{J}_2^M)=\mcal{I}_2^M$.

First, we show that ${\phi}(\mcal{J}_2^M)=\mcal{I}_2^M$. Let $G$ be any rank $1$ flat of $M$, and choose $i_0\in G$. We have that
\[{\phi}(h_G)= \left.\sum_{H\supseteq G}\right.^{F_U} (-_{F_U}x_H)  = \left.\sum_{i_0\in H}\right.^{F_U} (-_{F_U}x_H) = -_{F_U}\left(\left.\sum_{i_0\in H}\right.^{F_U}x_H\right)\in \mcal{I}_2^M\]
and thus ${\phi}(\mcal{J}_2^M)\subseteq\mcal{I}_2^M$. To see that the images of the ${\phi}(h_G)$ generate $\mcal{I}_2^M$, first take any $i_0\in E$ and choose the flat $G$ of rank $1$ containing $i_0$. \cref{lem:formal_inv} implies
\[{\phi}(h_G)=-_{F_U}\left(\left.\sum_{ i_0\in H}\right.^{F_U} x_H\right)=g_{i_0}\cdot \left(\left.\sum_{ i_0\in H}\right.^{F_U} x_H\right),\]
where $g_{i_0}\in \mcal{S}_M$ is an element with constant term $-1$. Now, \cref{lem:invertible} says the $g_{i_0}$ is a unit in $\mcal{S}_M$, so the images of the ${\phi}(h_G)$, with $\mrm{rank}_M(G)=1$, generate the ideal $\mcal{I}_2^M$. This tells us that ${\phi}(\mcal{J}_2^M)=\mcal{I}_2^M$.

Next, we show that ${\phi}(\mcal{J}_1^M)=\mcal{I}_1^M$. Observe that $(h_G-h_{G\vee H})(h_H-h_{G\vee H})$ is non-zero only when $H$ and $G$ are incomparable. Therefore, ${\phi}(\mcal{J}_1^M)$ is generated by the elements
\[({\phi}(h_G)-{\phi}(h_{G\vee H}))({\phi}(h_H)-{\phi}(h_{G\vee H})), \]
where $G$ and $H$ are incomparable flats of $M$. For $G$ and $H$ incomparable, the set of flats containing $G\vee H$ is a strict subset of both the set of flats containing $G$ and the set of flats containing $H$. Therefore, we can use \cref{lem:formal_inv} and \cref{lem:formal_dif} to write 
$({\phi}(h_G)-{\phi}(h_{G\vee H}))({\phi}(h_H)-{\phi}(h_{G\vee H})) $ 
as
\begin{equation}\label{eq:incomp} z_{G,H} \coloneqq \left( \sum_{G\subseteq G' \subsetneq G\vee H} x_{G'} g_{G'}\right)\left( \sum_{H\subseteq H' \subsetneq G\vee H} x_{H'} g_{H'}\right)\end{equation}
where each $g_{G'}$ and $g_{H'}$ is a formal power series over $\mbb{L}^*$ in the $x_G$ variables whose constant term is $-1$. Any flats $G',H'$ with $G\subseteq G' \subsetneq G\vee H$ and $H\subseteq H' \subsetneq G\vee H$ are forced to be incomparable since $G\vee H$ is the smallest flat containing both $G$ and $H$. Thus each $z_{G,H}$ is contained in $\mcal{I}_1^M$, giving us ${\phi}(\mcal{J}_1^M)\subseteq\mcal{I}_1^M$. 

To show $\mcal{I}_1^M\subseteq {\phi}(\mcal{J}_1^M)$, we will show that the $z_{G,H}$ generate $\mcal{I}_1^M$. As in the proof of \cite[Lemma $A.1$]{LLPP24}, we will prove, by induction on the sum of the coranks of incomparable flats $G$ and $H$, that $x_Gx_H$ is contained in the ideal ${\phi}(\mcal{J}_1^M)$. For the base cases, assume that $G$ and $H$ are incomparable flats whose coranks are each equal to one. In this case, $z_{G,H}= x_{G}x_{H}g_{G}g_{H}$. The polynomials $g_G$ and $g_H$ in $\mcal{T}_M$ have constant term $-1$, so by \Cref{lem:invertible}, they are units in $\mcal{T}_M$. Therefore $x_Gx_H = g_G^{-1}g_H^{-1}z_{G,H}\in {\phi}(\mcal{J}_1^M)$.

For general $G$ and $H$, each term in $z_{G,H}$ besides $x_Gx_Hg_Gg_H$ is divisible by a monomial $x_{G'}x_{H'}$ where $G'$ and $H'$ are incomparable flats whose sum of coranks is larger than the sum of the coranks of $G$ and $H$. Thus, our inductive hypothesis implies that $x_Gx_Hg_Gg_H\in {\phi}(\mcal{J}_1^M)$, and as before we conclude that $x_{G}x_H\in {\phi}(\mcal{J}_1^M)$.
\end{proof}

\begin{cor}\label{cor:free}
    The algebraic cobordism ring $\Omega^*(X_M)$ is a free $\mbb{L}^*$-module of finite rank.
\end{cor}
\begin{proof}
    As $\Omega^*(X_M)\simeq \mbb{L}^*\otimes_{\mbb{Z}}\CH^*(X_M)$, this follows from the fact that $\CH^*(X_M)$ is a free $\mbb{Z}$-module of finite rank (\cite[Corollary 1]{FY04}).
\end{proof}

We now note some formal consequences of the existence of an $\mbb{L}^*$-algebra isomorphism $\mbb{L}^*\otimes_{\mbb{Z}} \CH^*(X) \simeq \Omega^*(X)$ for a smooth variety $X$.
Recall \cref{dfn:free-theory}, which says that a free oriented cohomology theory $\mrm{h}_{(R,F)}^*$ gives rise to a formal group law $(R,F)$ such that $\mrm{h}_{(R,F)}^*(X)=\Omega^*(X)\otimes_\mbb{L^*} R$.  By definition, $\mrm{h}_{(R,F)}^*(\mrm{pt})=R$.

\begin{prop}\label{prop:exceptional-explanation}
    Let $\mrm{h}_{(R,F)}^*$ be a free oriented cohomology theory with formal group law $(R,F)$, and let $X$ be a smooth variety with an $\mbb{L}^*$-algebra isomorphism $\Theta_{(\mbb{L}^*,F_U)}\colon \Omega^*(X)\xrightarrow{\simeq} \mbb{L}^* \otimes_{\mbb{Z}} \CH^*(X)$. Then there is an induced $R$-algebra isomorphism
    $\Theta_{(R,F)}\colon \mrm{h}^*_{(R,F)}(X)\xrightarrow{\simeq } R\otimes_\mbb{Z} \CH^*(X)$ and, in particular, a ring isomorphism
    \[\Theta\colon K^0(X)\xrightarrow{\simeq} \CH^*(X). \]
\end{prop}
\begin{proof}   
    The $\mbb{L}^*$-algebra isomorphism $\Theta_{(\mbb{L}^*,F_U)}$ induces an $R$-algebra isomorphism
    \begin{align*} \Theta_{(\mbb{L}^*,F_U)}\otimes \mrm{id}_R\colon \Omega^*(X)\otimes_{\mbb{L}^*}R\xrightarrow{\simeq} &\left(\mbb{L}^*\otimes_{\mbb{Z}}\CH^*(X)\right)\otimes_{\mbb{L}^*} R,
    \end{align*}
    where the ring homomorphism $\mbb{L}^*\to R$ is induced by the formal group law $(R,F)$ (see \cref{rmk:Lazard-ring}).
    Noting the identifications of $R$-algebras 
    \[ \Theta_{(R,F)}'\colon \left(\mbb{L}^*\otimes_{\mbb{Z}}\CH^*(X)\right)\otimes_{\mbb{L}^*} R\xrightarrow{\simeq} \mrm{CH}^*(X)\otimes_{\mbb{Z}} (\mbb{L}^*\otimes_{\mbb{L}^*} R)\xrightarrow{\simeq}  R\otimes_{\mbb{Z}}\mrm{CH}^*(X),\] we define $\Theta_{(R,F)}$ to be $\Theta_{(R,F)}:=\Theta_{(R,F)}'\circ (\Theta_{(\mbb{L}^*,F_U)}\otimes \mrm{id}_R)$, which completes the proof of the first claim.

    It follows from the argument above that there is a $\mbb{Z}[\beta,\beta^{-1}]$-algebra isomorphism \[\Theta_{(\mbb{Z}[\beta,\beta^{-1}],F_M)}\colon \Omega^*(X)\otimes_{\mbb{L}^*}\mbb{Z}[\beta,\beta^{-1}]\xrightarrow{\simeq} \mbb{Z}[\beta,\beta^{-1}]\otimes_{\mbb{Z}}\mrm{CH}^*(X).\]
    The isomorphism $\Theta$ can be obtained from the one above by setting $\beta=1$, noting \cref{theo:universal-K} and Remark~\ref{rmk:k0fgl}. 
\end{proof}

By \Cref{prop:exceptional-explanation}, the following proposition holds if $X$ is a smooth variety such that  $\mbb{L}^*\otimes_{\mbb{Z}}\mrm{CH}^*(X)\simeq \Omega^*(X)$ and the Chow ring of $X$ has Poincar\'e duality.

\begin{prop}\label{prop:exceptionalHRR}
    Let $X$ be a smooth variety whose Chow ring vanishes in degree greater than $d$ and is equipped with an isomorphism $\mrm{deg}: \CH^{d}(X)\to \mbb{Z}$ such that for all $i$, the pairing
    \[\CH^i(X) \times \CH^{d-i}(X) \to \mbb{Z}, \quad (x,y) \mapsto \mrm{deg}(xy)\]
    is perfect. Suppose there is a ring isomorphism $\zeta: K^0(X)\to \CH^*(X)$ and let $\chi:K^0(X)\to \mbb{Z}$ be any group homomorphism. Then there is an element $f\in \CH^*(X)$ such that
    \[\chi(x) = \mrm{deg}(\zeta(x)f), \quad \text{for all } x\in K^0(X). \]
\end{prop}
\begin{proof}
    Consider the group homomorphism $\chi \circ \zeta^{-1} \colon \mrm{CH}^i(X)\to \mbb{Z}$. As $\mrm{deg}$ induces a perfect pairing, there is an element $t_{d-i}\in \mrm{CH}^{d-i}(X)$ satisfying the equality $\chi\circ \zeta^{-1}(x) = \mrm{deg}(xt_{d-i})$. Taking $f$ to be the sum of all such $t_{d-i}$s yields the result.
\end{proof}

If $X$ is a smooth projective variety, the structure map $\pi:X\to \mrm{pt}$ induces a group homomorphism $\chi\coloneqq \pi_*: K^0(X) \to \mbb{Z}$ known as the Euler characteristic map. If, in addition, there is a ring isomorphism $\zeta: K^0(X)\to \CH^*(X)$, then \Cref{prop:exceptionalHRR} tells us that there is an element $f\in \CH^*(X)$ such that $\chi(y)=\mrm{deg}(\zeta(y)f)$. The relationship between $\zeta$ and $f$ is analogous to the relationship between the Chern character map and the Todd class in the Hirzebruch--Riemann--Roch formula. Despite this analogy, we note that $\zeta$ is not equal to the Chern character map and $f$ is not equal to the Todd class of $X$.

Combining \Cref{thrm:simplicial}, \Cref{prop:exceptional-explanation}, and \Cref{prop:exceptionalHRR}, we obtain the following corollary. This corollary appears, in a slightly stronger form, as Theorem 1.2 of \cite{LLPP24} and is known as the exceptional isomorphism for $W_{\mcal{H}}$.

\begin{cor}\label{cor:exceptional-explanation}
    There is an integral ring isomorphism $\zeta \colon K^0(X_M)\to \CH^*(X_M)$ defined by sending the simplicial generator $h_G\in K^0(X_M)$ to the corresponding simplicial generator $h_G\in \CH^*(X_M)$. If $M_\mcal{H}$ is the matroid of a complex hyperplane arrangement $\mcal{H}$, then there is an element $f\in \CH^*(W_{\mcal{H}})$ such that 
    \[\chi(x) = \deg(\tilde\zeta(x)f),\quad  \text{for all } x\in K^0(W_{\mcal{H}}) \]
    where $\chi$ is the Euler characteristic map and $\tilde \zeta$ is the composition of ring isomorphisms \[\tilde\zeta:K^0(W_{\mcal{H}}) \xrightarrow{(\iota^*)^{-1}} K^0(X_{M_\mcal{H}}) \xrightarrow{\zeta} \CH^*(X_{M_{\mcal{H}}}) \xrightarrow{\iota^*} \CH^*(W_{\mcal{H}}).\]
\end{cor}

We note to the reader that in Theorem 1.2 of \cite{LLPP24} it is proven that, in the situation of the above corollary, $f$ is equal to $1+h_E+h_E^2+\ldots \in \CH^*(W_\mcal{H})$. This explicit expression for $f$ is useful for calculations and does not follow abstractly from our \Cref{thrm:simplicial}.
  
The isomorphism $\Omega^*(X_M)\simeq \mbb{L}^*\otimes \CH^*(X_M)$ is quite special and does not hold for many varieties, even in the case of smooth, projective toric varieties. The following example, due to Speyer \cite{Speyer}, illustrates this phenomenon. Note that, for a smooth variety $X$, an isomorphism of $\mbb{L}^*$-algebras $\Omega^*(X)\simeq \mbb{L}^*\otimes \CH^*(X)$ would imply a ring isomorphism $K^0(X)\simeq \CH^*(X)$.

\begin{egg}\label{egg:smooth-projective-toric}[A smooth projective toric variety $X_{\Sigma}$ where $K^0(X_{\Sigma})\not\simeq \CH^*(X_{\Sigma})$]
Let $\Sigma \subseteq \mbb{R}^3$ be the simplicial fan whose rays have primitive vectors equal to the columns of the matrix
\[
\bordermatrix{ & a & b & d & A & B & D \cr
       & 1 & 0 & 0 & -1 & 0 & 0 \cr
       & 0 & 1 & 0  & -1 & -1 & 0\cr
       & 0 & 0 & 1 & 0 & -2 & -1 }
\]
and whose maximal cones correspond to the subsets $S\subseteq \{a,A, b, B, d,D\}$ such that $|S|=3$ and $S$ does not contain any of the sets $\{a,A\}$, $\{b,B\}$, $\{d,D\}$. The fan $\Sigma$ is projective, unimodular and combinatorially equivalent to the normal fan of the cube. The associated toric variety $X_{\Sigma}$ is smooth and projective. Let 
\[ R_{\Sigma} = \frac{\mbb{Z}[x_{a},x_{A},x_b,x_{B},x_d,x_{D}]}{\left\langle x_ax_{A},x_bx_{B},x_dx_{D}\right\rangle} \]
be the Stanley--Reisner ring of $\Sigma$. By \cite[Section 5.2]{fultonToric} and \cite[Theorem 1]{CCKR25}, we can compute ring presentations of $\CH^*(X_{\Sigma})$ and $K^0(X_{\Sigma})$ as quotients of $R_{\Sigma}$. Their respective ring presentations are
\begin{gather*}
    \CH^*(X_{\Sigma}) \simeq R_{\Sigma}/\left\langle x_a-x_{A}, x_b-x_{A}-x_{B}, x_d-2x_{B}-x_{D} \right\rangle\\
    K^0(X_{\Sigma}) \simeq R_{\Sigma}/\left\langle x_a-x_{A}, (1-x_b) - (1-x_A)(1-x_B), (1-x_d) - (1-x_B)^2(1-x_D) \right\rangle.
\end{gather*}
Let $\mbb{F}_2= \mbb{Z}/2\mbb{Z}$ be the field with two elements. We show that the $\mbb{F}_2$-algebras $\CH^*(X_{\Sigma})\otimes \mbb{F}_2$ and $K^0(X_{\Sigma}) \otimes \mbb{F}_2$ are not isomorphic (as they would be had we tensored with $\mbb{Q}$) and hence that $\CH^*(X_{\Sigma})\not\simeq K^0(X_{\Sigma})$. 
One can check, for instance using the computer program Macaulay2 \cite{M2} (see \cite{EP26} for our code used to verify this example), that the sets
\begin{equation}\label{eq:bases}
\{1,x_A,x_B,x_D,x_Ax_D,x_Bx_D,x_B^2,x_B^2x_D\}\quad\text{and}\quad \{1,x_A,x_b,x_B,x_d,x_D,x_Ax_D,x_Bx_D\} 
\end{equation}
are $\mbb{F}_2$-bases of $\CH^*(X_{\Sigma})\otimes \mbb{F}_2$ and $K^0(X_{\Sigma})\otimes \mbb{F}_2$, respectively. Over $\mbb{F}_2$, squaring is a linear map and thus the spaces of squares in $\CH^*(X_{\Sigma})\otimes \mbb{F}_2$ and $K^0(X_{\Sigma})\otimes \mbb{F}_2$ are spanned by the squares of the bases given in (\ref{eq:bases}). We calculate that 
\[\{1,x_B^2\} \quad\text{and}\quad \{1,x_b+x_A+x_B,x_b+x_d+x_A+x_B+x_D\} \]
are the non-zero squares of these bases. The dimensions of these spaces are different and hence $\CH^*(X_{\Sigma})\otimes \mbb{F}_2$ is not isomorphic to $K^0(X_{\Sigma})\otimes \mbb{F}_2$ as an $\mbb{F}_2$-algebra. Thus, $\mrm{CH}^*(X_\Sigma)$ and $K^0(X_\Sigma)$ are not isomorphic as rings.
\end{egg}

\section{Algebraic and complex cobordism rings of $W_{\mcal{H}}$}\label{cobordism-of-wonderful}
    
    Let $\mcal{H}$ be a hyperplane arrangement in a complex vector space $L$ whose hyperplanes are indexed by a finite set $E$, and let $W_\mcal{H}$ be the associated wonderful variety. We denote the matroid associated with $\mcal{H}$ by $M_{\mcal{H}}$ and the toric variety of the Bergman fan of this matroid by $X_{M_\mcal{H}}$. In \cref{subsection:liftings}, we recall some comparison results between algebraic cobordism and Chow theory, and between  complex cobordism and singular cohomology.
    In \cref{subsection:cobordismWH}, we apply these comparison results to $W_{\mcal{H}}$ to obtain $\mbb{L}^*$-module generators for its algebraic and complex cobordism rings.
    In \cref{subsection:ring-presentation}, we prove that the pullback of the inclusion $W_\mcal{H}\hookrightarrow X_{M_{\mcal{H}}}$ induces an $\mbb{L}^*$-algebra isomorphism $\Omega^*(X_{M_\mcal{H}})\simeq \Omega^*(W_\mcal{H}) $, and that the natural map $\vartheta\colon \Omega^*\to MU^*$ induces an $\mbb{L}^*$-algebra isomorphism $\Omega^*(W_{\mcal{H}})\simeq MU^*(W_\mcal{H}(\mbb{C}))$. In \cref{subsection:the-cobordism-class-of-wonderful}, we show that the algebraic cobordism class $[W_{\mathcal{H}}\to \mrm{pt}]\in \Omega^*(\mrm{pt})$ depends only on $M_\mcal{H}$.

\subsection{Liftings to cobordism}\label{subsection:liftings}
    In our proof of Theorem~\ref{theo:pullback-in-cobordism}, it will be useful to have an explicit $\mbb{L}^*$-module generating set for the algebraic and complex cobordism rings of $W_\mcal{H}$. We collect two general theorems which let us lift generators from the Chow ring and singular cohomology ring of a space to its algebraic and complex cobordism rings. The first such result appears as Lemma 4.1 in unpublished work of Deshpande \cite{D09}. For completeness, we include a proof.
    \begin{lem}(cf. \cite[Lemma 4.1]{D09})\label{lem:lifts-generate}
        Let $X\in \textbf{Sm}_{\mbb{C}}$. Suppose that $x_1,\dotsc,x_n$ are homogeneous elements in $\mrm{CH}^*(X)$ that generate $\mrm{CH}^*(X)$ as an abelian group. If $y_1,\dotsc,y_n$ are homogeneous elements in $\Omega^*(X)$ that map to $x_1,\dotsc, x_n$ under the universal morphism $\Omega^*(X)\to \CH^*(X)$, then $y_1,\dotsc,y_n$ generate $\Omega^*(X)$ as an $\mbb{L}^*$-module. 
    \end{lem}
    \begin{proof}
    Recall the isomorphism of rings from \cref{theo:universal-CH}, \begin{equation}\label{eqn:universal-CH}
        \Omega^*(X)\otimes_{\mbb{L}^*}\mbb{Z}\xrightarrow{\sim}\mrm{CH}^*(X),
    \end{equation}
    which is induced by the universal morphism of oriented cohomology theories $\Omega^*\to \mrm{CH}^*$ of \cref{theo:universal}. Let $y\in\Omega^*(X)$ be homogeneous of degree $d$. Consider the image $\overline{y}$ of $y$ in $\mrm{CH}^*(X)$ under the isomorphism (\ref{eqn:universal-CH}). As $x_1,\dotsc,x_n$ generate $\mrm{CH}^*(X)$ as an abelian group, there are integers $\lambda_{1},\dotsc,\lambda_n\in\mbb{Z}$ such that $\overline{y}=\sum_{i=1}^n \lambda_i x_i.$
    Therefore, 
    \[y-\sum_{i=1}^n \lambda_iy_i\in \mbb{L}^{< 0}\cdot \Omega^*(X).\]
    In particular, there are homogeneous elements $\ell_i\in\mbb{L}^{< 0}$ and homogeneous elements $q_i\in\Omega^*(X)$, such that 
    \begin{equation}\label{eqn:main-eqn}
        y-\sum_{i=1}^n \lambda_iy_i=\ell_1 q_1+\dotsb+\ell_t q_t
    \end{equation}
    As $y$ is homogeneous of degree $d$ and each $y_i$ is homogeneous in $\Omega^*(X)$, we may assume that $\lambda_i=0$ whenever $\mrm{deg}(y_i)\neq d$ and that $\mrm{deg}(\ell_i q_i)=d$. As each $\ell_i$ has degree strictly less than $0$ in $\Omega^*(X)$, it follows that each $q_i$ must have degree strictly greater than $d$ in $\Omega^*(X)$.

    Our goal is to express $y$ as an $\mbb{L}^*$-linear combination of the $y_i$s. We will proceed by decreasing induction on $d$, beginning with the base cases of $d>\dim_{\mbb{C}}(X)$ and $d=\dim_{\mbb{C}}(X)$. As $\Omega^*(X)$ vanishes in degrees larger than $\dim_\mbb{C}(X)$, if $d>\mrm{dim}_\mbb{C}(X)$, then $\Omega^d(X)=(0)$. Thus, $y=0$ in this case. If $d=\mrm{dim}_\mbb{C}(X)$, then, as $\mrm{deg}(q_i)>d$, we must have that each $q_i=0$. So Equation (\ref{eqn:main-eqn}) implies that $y=\sum_{i=1}^n\lambda_i y_i$ in this case, which completes our base cases.
    
    For our inductive hypothesis, we assume that every homogeneous element in $\Omega^*(X)$ of degree greater than $d$ is expressible as an $\mbb{L}^*$-linear combination of the $y_i$s. As each $q_i$ in Equation (\ref{eqn:main-eqn}) is homogeneous of degree $>d$, each $q_i$ is expressible as an $\mbb{L}^*$-linear combination of the $y_i$s. Therefore, Equation (\ref{eqn:main-eqn}) says that $y$ is an $\mbb{L}^*$-linear combination of the $y_i$s. This completes the inductive step.

    As every homogeneous element in $\Omega^*(X)$ is a $\mbb{L}^*$-linear combination of the $y_i$s, and as every element in $\Omega^*(X)$ is an $\mbb{L}^*$-linear combination of homogeneous elements, it follows that every element in $\Omega^*(X)$ is an $\mbb{L}^*$-linear combination of the $y_i$s. Thus, $y_1,\dotsc,y_n$ generate $\Omega^*(X)$ as a $\mbb{L}^*$-module.
    \end{proof}

    The previous proof relied on the ring isomorphism $\Omega^*(X) \otimes_{\mbb{L}^*} \mbb{Z}\simeq \CH^*(X)$ of Theorem~\ref{theo:universal-CH}. In general, there is not an analogous ring isomorphism $H^*(X(\mbb{C}))\simeq MU^*(X(\mbb{C}))\otimes_{\mbb{L}^*}\mbb{Z}$ for a smooth variety $X$; see the discussion in the introduction of \cite{T97}. However, under additional hypotheses, one can use the ``Atiyah--Hirzebruch spectral sequence'' to lift generators from $H^*(X(\mbb{C}))$ to $MU^*(X(\mbb{C}))$. We refer the reader to the book of Stong \cite[Page 144]{Stong} for a proof\footnote{In \cite{Stong}, Stong works with the homological analog of complex cobordism, complex bordism $MU_*(X(\mbb{C}))$. As $X$ is a smooth, complex projective variety, there is a Poincar\'e duality isomorphism $MU_i(X(\mbb{C}))\simeq MU^{2n-i}(X(\mbb{C}))$ \cite[Chapter 14]{S02}. Our theorem statement follows from Stong's by an application of Poincar\'e duality.}.

    \begin{lem}(cf. \cite[Page 144]{Stong})\label{lem:ahssApplication}
         Let $X$ be an $n$-dimensional smooth, complex projective variety. If the integral cohomology ring of $X$ has no torsion, then $MU^*(X(\mbb{C}))$ is a free $\mbb{L}^*$-module that is isomorphic to \[\mbb{L}^*\otimes_{\mbb{Z}}H^*(X(\mbb{C}))\]
         as an $\mbb{L}^*$-module. If $g_1,\ldots , g_r$ are homogeneous elements in $H^*(X(\mbb{C}))$ that form a $\mbb{Z}$-basis for $H^*(X(\mbb{C}))$, and $\tilde{g}_1, \ldots, \tilde{g}_r$ are homogeneous elements in $MU^*(X)$  such that $\tilde{g}_i\mapsto g_i$ under the universal morphism of complex-oriented cohomology theories $MU^*(X(\mbb{C}
        ))\to H^*(X(\mbb{C}
        ))$,
        then $\tilde{g}_1, \ldots, \tilde{g}_r$ form an $\mbb{L}^*$-module basis of $MU^*(X(\mbb{C}
        ))$.
    \end{lem}

\subsection{The $\mbb{L}^*$-module structure of the cobordism rings of $W_{\mcal{H}}$}\label{subsection:cobordismWH}
    After a brief interlude on the cycle class map, we apply the results of Section~\ref{subsection:liftings} to $W_{\mcal{H}}$. 
    If $X$ is a smooth, projective algebraic variety over $\mbb{C}$, then every closed subvariety $Y$ of $X$ defines a class in the even degree part of the singular cohomology ring $H^{2*}(X(\mbb{C}))$.
    There is a degree-doubling ring homomorphism, called the \textit{cycle class map}, from the Chow ring $\CH^*(X)$ to the singular cohomology ring $H^*(X(\mbb{C}))$, which sends the class of a subvariety $Y\subseteq X$ in $\CH^*(X)$ to the class $[Y(\mbb{C})]$ that it defines in $H^{2*}(X)$ (see, for example, \cite[Section 1.1.2]{V16}). 

    \begin{lem}(\cite[Theorem 1 of the Appendix]{K92})\label{lem:cycle-blowup}
        Suppose that $Y$ is a variety for which the cycle map is an isomorphism, $X\hookrightarrow Y$ is a regularly embedded subvariety, and the cycle class map is an isomorphism for $X$. Then the cycle class map is an isomorphism for the blowup of $Y$ along $X$.
    \end{lem}

    The following lemma is well known to experts. As we were unable to find a reference, we give a proof. The idea for this proof was communicated to us by Matt Larson. 
    \begin{lem}\label{lem:cycle-class}
        The cycle class map $c_{\mcal{H}}\colon \CH^*(W_{\mcal{H}})\to H^{2*}(W_{\mcal{H}}(\mbb{C}))$ is an isomorphism.
    \end{lem}
    \begin{proof}
        Let $G_1\succ G_2 \succ \cdots \succ G_k$ be a total order of the proper non-empty flats of $M_{\mcal{H}}$ that refines the partial order by containment. Recall that the variety $W_{\mcal{H}}$ is constructed by a sequence of blowups 
        \[W_{\mcal{H}}^{G_0} \coloneqq \mbb{P}(L) \leftarrow W_{\mcal{H}}^{G_1} \leftarrow W_{\mcal{H}}^{G_2}\leftarrow \cdots \leftarrow  W_{\mcal{H}}^{G_k} = W_{\mcal{H}}\]
        where $W_{\mcal{H}}^{G_i}$ is obtained from $W_{\mcal{H}}^{G_{i-1}}$ by blowing up the strict transform of $\mbb{P}(L_{G_i})$ in $W_{\mcal{H}}^{G_{i-1}}$. By a repeated use of the blow-up closure lemma \cite[Chapter 22]{Vakil}, one can verify that the strict transform of $\mbb{P}(L_G)$ is isomorphic to the wonderful compactification of the hyperplane arrangement $\mbb{P}\mcal{H}^G$ in $\mbb{P}(L_G)$ induced by intersecting $\mbb{P}\mcal{H}$ with $\mbb{P}(L_G)$. Our claim now follows by induction on the dimension of $W_{\mcal{H}}$ and \Cref{lem:cycle-blowup}.
    \end{proof}
    
    The next results follow from \Cref{lem:cycle-class} and Feichtner and Yuzvinsky's work comparing the Chow ring of $X_{M_\mcal{H}}$ and the singular cohomology ring of $W_{\mcal{H}}(\mbb{C})$ \cite{FY04}.
    
    \begin{theo}(\cite[Corollary 1]{FY04})\label{theo:FY-basis}
        For a proper, non-empty flat $G$ of $M_{\mcal{H}}$, let $D_G$ be the divisor defined in \Cref{defn:divisor-in-W}. The rings $\CH^*(W_{\mcal{H}})$ and $H^*(W_{\mcal{H}}(\mbb{C}))$ are free $\mbb{Z}$-modules generated by products of the classes $[D_G]$ and $[D_G(\mbb{C})]$, respectively.
    \end{theo}

    \begin{theo}\label{theo:iso-pullback}(\cite[Corollary 2]{FY04})
        The pullback map $\iota^*\colon \CH^*(X_{M_\mcal{H}})\to \CH^*(W_\mcal{H})$ induced by the inclusion $ \iota \colon W_\mcal{H}\hookrightarrow X_{M_\mcal{H}}$ is an isomorphism of rings.
    \end{theo}

    We now apply the results of Section~\ref{subsection:liftings} to the wonderful compactification $W_\mcal{H}$.
    \begin{prop}\label{lem:cobordism-lift}
        There are finite sets of products of divisors classes $[D_G\to W_\mcal{H}]$ and $[D_G(\mbb{C})\to W_{\mcal{H}}(\mbb{C})]$ that generate $\Omega^*(W_{\mcal{H}})$ and $MU^*(W_{\mcal{H}}(\mbb{C}))$ as $\mbb{L}^*$-modules. The ring $MU^*(W_{\mcal{H}}(\mbb{C}))$ is a free $\mbb{L}^*$-module.
    \end{prop}
    \begin{proof}
        As $D_G$ is a smooth divisor in $W_{\mcal{H}}$, $[D_G\to W_\mcal{H}]$ and $[D_G(\mbb{C})\to W_{\mcal{H}}(\mbb{C})]$ are the first Chern classes of $\mcal{O}_{W_{\mcal{H}}}(D_G)$ in $\Omega^*(W_{\mcal{H}})$ and $MU^*(W_{\mcal{H}}(\mbb{C}))$, respectively. Similarly, the first Chern classes of $\mcal{O}(D_G)$ in $\CH^*(W_{\mcal{H}})$ and $H^*(W_{\mcal{H}}(\mbb{C}))$ are equal to $[D_G]$ and $[D_G(\mbb{C})]$. As a morphism of cohomology theories sends Chern classes to Chern classes, the claim now follows from \Cref{theo:FY-basis} and Lemmas \ref{lem:lifts-generate} and \ref{lem:ahssApplication}. 
    \end{proof}

\subsection{The isomorphisms}\label{subsection:ring-presentation}
    We prove in \cref{theo:pullback-in-cobordism} of this subsection that the inclusion $\iota\colon W_\mcal{H}\hookrightarrow X_{M}$ and universal morphism $\vartheta\colon \Omega^*\to MU^*$ induce $\mbb{L}^*$-algebra isomorphisms $\Omega^*(X_M)\simeq \Omega^*(W_\mcal{H})$ and $\Omega^*(W_\mcal{H})\simeq MU^*(W_\mcal{H}(\mbb{C}))$.
    
    \begin{prop}\label{prop:map:omega-MU}     
        The universal morphism $\vartheta\colon \Omega^*\to MU^{2*}$ induces a \textit{surjection} of $\mbb{L}^*$-algebras 
        \[\vartheta_{W_\mcal{H}}\colon \Omega^*(W_\mcal{H})\to MU^{2*}(W_\mcal{H}(\mbb{C})).\]
    \end{prop}
    \begin{proof}
        By \cref{lem:cobordism-lift}, we can choose a finite set of homogeneous elements $h_1,\dotsc,h_k$ that generate $\Omega^*(W_\mcal{H})$ as an $\mbb{L}^*$-module. Say that $h_i\mapsto g_i$ under the natural map $\Omega^*(W_\mcal{H})\to \CH^*(W_\mcal{H})$. Then the $g_i$ generate $\CH^*(W_\mcal{H})$ as a $\mbb{Z}$-module by surjectivity of the natural map $\Omega^*(W_\mcal{H})\to \CH^*(W_\mcal{H})$. 
        In addition, let $\widetilde{g}_i$ be the image of $h_i$ under the natural map $\Omega^*(W_\mcal{H})\to MU^{2*}(W_\mcal{H}(\mbb{C}))$. As the natural maps $\Omega^*(W_\mcal{H})\to \CH^*(W_\mcal{H})$ and $\Omega^*(W_\mcal{H})\to MU^{2*}(W_\mcal{H}(\mbb{C}))$ are homomorphisms of graded rings, it follows that the $g_i$ are homogeneous in $\CH^*(W_\mcal{H})$ and the $\widetilde{g}_i$ are homogeneous in $MU^{2*}(W_\mcal{H}(\mbb{C}))$.       
        As explained in \cite[Section 6]{HK11}, the following diagram commutes:
        \[
        \begin{tikzcd}
            \Omega^*(W_\mcal{H}) \arrow[r,"\vartheta_{W_\mcal{H}}"] \arrow[d] & MU^{2*}(W_\mcal{H}(\mbb{C})) \arrow[d]\\
            \CH^*(W_\mcal{H}) \arrow[r,"c_\mcal{H}"] & H^{2*}(W_\mcal{H}(\mbb{C}).)
        \end{tikzcd}
        \]
        By commutativity of the diagram and the fact that $c_\mcal{H}$ is an isomorphism of graded rings, it follows from \cref{lem:ahssApplication} that the $\widetilde{g}_i$ generate $MU^{2*}(W_\mcal{H}(\mbb{C}))$ as an $\mbb{L}^*$-module. Thus, the map $\Omega^*(W_\mcal{H})\to MU^{2*}(W_\mcal{H}(\mbb{C}))$ is surjective.
    \end{proof}

     \begin{lem}\label{lem:pullback-of-inclusion}
        The pullback map $\iota^*\colon\Omega^*(X_M)\to\Omega^*(W_\mcal{H})$ induced by the inclusion $\iota\colon W_\mcal{H}\hookrightarrow X_M$ is surjective. 
    \end{lem}
    \begin{proof}
        By \cref{prop:FY-presentation}, $\Omega^*(X_M)$ is generated as an $\mbb{L}^*$-algebra by the elements $[D_{\rho_G}\to X_M]$ for $G$ a proper non-empty flat of $M$. It follows that $\iota^*$ is defined by where it sends these generators. From \Cref{lem:smooth_pullbacks} and \Cref{lem:intersection}, we conclude that $\iota^*([D_{\rho_G}\hookrightarrow X_M])=[D_G \hookrightarrow W_{\mcal{H}}]$. \Cref{lem:cobordism-lift} implies that the classes $[D_G\hookrightarrow W_\mcal{H}]$ generate $\Omega^*(W_\mcal{H})$ as an $\mbb{L}^*$-algebra. From this, we conclude surjectivity of $\iota^*$. %
    \end{proof}

    \begin{theo}\label{theo:pullback-in-cobordism}
        Both $\iota^*\colon \Omega^*(X_{M_\mcal{H}})\to \Omega^*(W_\mcal{H})$ and $\vartheta\colon \Omega^*(W_\mcal{H})\to MU^{2*}(W_\mcal{H}(\mbb{C}))$ are isomorphisms of $\mbb{L}^*$-algebras.
    \end{theo}
    \begin{proof}
        Consider the composition of $\mbb{L}^*$-algebra maps:
        \[
    \begin{tikzcd}
    \tau\colon \Omega^*(X_M)\arrow[r,"\iota^*"] & \Omega^*(W_\mcal{H})\arrow[r,"\vartheta"] & MU^{2*}(W_\mcal{H}(\mbb{C})).
    \end{tikzcd}
    \] 
       By \cref{prop:map:omega-MU} and \cref{lem:pullback-of-inclusion}, both $\iota^*$ and $\vartheta$ are surjective maps of $\mbb{L}^*$-algebras, so the composition $\tau$ is surjective as well. We know from \cref{cor:free} and \cref{lem:cobordism-lift} that $\Omega^*(X_M)$ and $MU^{2*}(W_\mcal{H}(\mbb{C}))$ are free $\mbb{L}^*$-modules of the same rank, with that rank being the $\mbb{Z}$-module rank of $\CH^*(W_\mcal{H})$. As $\mbb{L}^*$ is an integral domain and $\tau$ is a surjective $\mbb{L}^*$-algebra morphism between free $\mbb{L}^*$-modules of the same rank, $\tau$ must be an isomorphism. In particular, $\tau=\vartheta\circ \iota^*$ is injective, which means that $\iota^*$ is injective. As $\iota^*$ is both injective and surjective, it must be an isomorphism. Thus, $\vartheta=\tau\circ (\iota^*)^{-1}$ is an isomorphism as well.
    \end{proof}

\begin{rmk}[The complex cobordism of $X_{M_{\mathcal{H}}}$]\label{rmk:binder}
Consider the following commutative diagram of rings
\begin{equation}\label{eq:cobordism}
        \begin{tikzcd}
            \Omega^*(X_{M_\mcal{H}}) \arrow[r,"\vartheta_{M_\mcal{H}}"]  \arrow[d,"\iota^{\Omega^*}"] &
            MU^{2*}(X_{M_\mcal{H}}(\mbb{C})) \arrow[d,"\iota^{MU^{*}}"] \\
            \Omega^*(W_\mcal{H}) \arrow[r,"\vartheta_{W_\mcal{H}}"] &
            MU^{2*}(W_\mcal{H}(\mbb{C})) 
        \end{tikzcd}
    \end{equation}
    where the horizontal arrows are induced by the universal morphism $\vartheta: \Omega^*\to MU^{2*}$ of oriented cohomology theories and the vertical arrows are induced by the inclusion map $\iota: W_{\mcal{H}}\to X_{M_{\mcal{H}}}$. We have proven that $\vartheta_{W_\mcal{H}} \circ \iota^{\Omega^*}$ is a ring isomorphism. This implies that $\vartheta_{M_{\mcal{H}}}$ is injective and $\iota^{MU^*}$ is surjective. However, we do not know whether $\vartheta_{M_{\mcal{H}}}$ and $\iota^{MU^*}$ are ring isomorphisms. In the analogous diagram
    \begin{equation}\label{eq:cohomology}
        \begin{tikzcd}
            \CH^*(X_{M_\mcal{H}}) \arrow[r,"c_{M_\mcal{H}}"]  \arrow[d,"\iota^{\CH^*}"] &
            H^{2*}(X_{M_\mcal{H}}(\mbb{C})) \arrow[d,"\iota^{H^{*}}"] \\
            \CH^*(W_\mcal{H}) \arrow[r,"c_{W_\mcal{H}}"] &
            H^{2*}(W_\mcal{H}(\mbb{C})) 
        \end{tikzcd}
    \end{equation}
    we have that $\iota^{\CH^*}$ and $c_{W_{\mcal{H}}}$ are always ring isomorphisms but $c_{M_{\mcal{H}}}$ and $\iota^{H^*}$ can fail to be ring isomorphisms. This is because the singular cohomology of $X_{M_\mcal{H}}$ can have non-zero even cohomology that is not of Hodge--Tate type.  For example, in work of Binder \cite{Binder26}, it is shown that if $\mcal{H}$ is an arrangement of $4$ points in $\mbb{P}^1$, then $\mrm{Gr}_6 H^4(X_{M_\mcal{H}}) = \mbb{Q}^2$. In particular, $H^4(X_{M_{\mcal{H}}}(\mbb{C}))$ is non-zero. However, $W_{\mcal{H}}\simeq \mbb{P}^1$ so the sequence of maps 
     \[\CH^*(W_{\mcal{H}}) \to H^{2*}(X_{M_{\mcal{H}}}(\mbb{C})) \to H^{2*}(W_{\mcal{H}}(\mbb{C}))\]
    cannot consist of ring isomorphisms. 
    
    We note to the reader that \cref{eq:cohomology} does not immediately tell us information about \cref{eq:cobordism}. As the formal group laws of singular cohomology and Chow theory are additive, one can try to tensor \cref{eq:cobordism} with $\mbb{Z}$ via the map $\mbb{L}^*\to \mbb{Z}$ induced by the additive formal group law. This successfully recovers three of the terms of \cref{eq:cohomology}. However, as $X_{M_\mcal{H}}(\mbb{C})$ is not compact, we do not know how $MU^*(X_{M_\mcal{H}}(\mbb{C})) \otimes_{\mbb{L}^*} \mbb{Z}$ compares to $H^*(X_{M_\mcal{H}}(\mbb{C}))$. The comparison between the functors $MU^*(-)\otimes_{\mbb{L}^*} \mbb{Z}$ and $H^*(-)$ is quite subtle; see \cite{T97}.
    
    In \cite{Binder26}, Binder extensively studies the singular cohomology rings $H^*(X_M)$ of matroids \cite{Binder26}. It would be interesting to see if one can extend their work to the complex cobordism rings $MU^*(X_M)$ of matroids.
\end{rmk}

\subsection{On the cobordism class of the wonderful variety}\label{subsection:the-cobordism-class-of-wonderful}
    As $W_\mathcal{H}$ is a smooth projective complex variety, it determines a class $[W_{\mathcal{H}}\to \mrm{pt}]\in \Omega^*(\mrm{pt})$. In this subsection, we show that this class only depends on the matroid $M_\mcal{H}$ of $H$. We first recall the classical classification of classes in $MU^*(\mrm{pt})$ by their Chern numbers \cite{Q71}.

\begin{defn}
    Let $X$ be a compact complex manifold of complex dimension $n$. Let $\lambda=(\lambda_1,\dotsc,\lambda_r)$ be a partition of $n$, i.e., $\lambda$ is a sequence of weakly increasing positive integers such that $\lambda_1+\lambda_2+\dotsb+ \lambda_r=n$. The \textbf{$\lambda$-th Chern number} of $X$ is defined to be the integer
    \[c_\lambda(X):=\langle c_{\lambda_1}(T_X)\dotsb c_{\lambda_r}(T_X),[X]\rangle, \]
    where $T_X$ is the tangent bundle of $X$, $[X]$ is the fundamental class of $X$ in the homology group $H_*(X)$, and $\langle \text{ $,$ }\rangle$ is the Poincar\'{e} duality pairing between $H^*(X)$ and $H_*(X)$. The \textbf{total Chern class} $c(E)$ of a rank $r$ vector bundle $E$ on $X$ is defined to be the element $c(E):=\sum_{i=0}^{r}c_i(E)\in H^*(X)$.
\end{defn}
    
\begin{theo}(\cite[Theorem 6.5]{Q71})\label{theo:cobordant-manifolds}
    Suppose $X_1$ and $X_2$ are two compact complex manifolds. The classes $[X_1\to \mrm{pt}]$ and $[X_2\to \mrm{pt}]$ in $MU^*(\mrm{pt})$ are equal if and only if $X_1$ and $X_2$ have the same Chern numbers.
\end{theo}

\begin{rmk}
    Let $X$ be a compact complex manifold of complex dimension $n$. The Chern number $c_{(n)}(X)= \langle c_n(T_X), [X]\rangle  $ is equal to the Euler characteristic of $X$. If $X=W_\mcal{H}(\mbb{C})$, then, as $W_\mcal{H}(\mbb{C})$ has no odd cohomology, $c_{(n)}(W_\mcal{H})= \chi(W_\mcal{H}) = \dim H^*(W_{\mcal{H}}(\mbb{C}), \mbb{Q})$ which is an invariant of the matroid $M_\mcal{H}$ of $W_\mcal{H}$. The combinatorial study of this invariant is of recent interest \cite{FMSV24}.
\end{rmk}

The total Chern class of $T_{W_\mathcal{H}}$ was recently given an explicit formula which depends only on the matroid $M_\mcal{H}$ in the work of Cheng \cite{Cheng25}. The total Chern class of $T_{W_\mathcal{H}}$ is closely related to the Chern classes of the tautological bundle $\mathcal{Q}_{\mcal{H}}$ on the permutahedral variety \cite{BEST} and the Poincar\'e dual of the Chern-Schwartz-MacPherson class of the complement $W_{\mathcal{H}}\setminus \mathcal{H}$ \cite{AL24}. We now recall Cheng's formula using the presentation for $H^{2*}(W_{\mcal{H}}(\mbb{C}))$ given by \Cref{rmk:chow} and the cycle class isomorphism $\CH^*(W_{\mcal{H}})\simeq H^{2*}(W_{\mcal{H}}(\mbb{C}))$. This presentation agrees with the one given by Feichtner and Yuzvinsky in \cite{FY04}.

\begin{defn}
    For an integer $k\geq 0$, let \[S_{k,M_\mcal{H}} := \sum_{G \text{ a rank $r-k$ flat of $M_\mcal{H}$}} x_G \in H^2(W_{\mcal{H}}(\mbb{C})).\]
    Note that $S_{0,M_\mcal{H}} = x_E$, which is the class denoted by $-\alpha$ in \cite{AHK18}.
\end{defn}

\begin{theo}[Theorem 3.4 of \cite{Cheng25}]\label{theo:Cheng}
The total Chern class of the tangent bundle of $W_{{\mcal{H}}}(\mbb{C})$ is equal to 
\[c(T_{W_{\mcal{H}}(\mbb{C})})= \prod_{i=1}^{r-1} (1+S_{i,M_\mcal{H}}) \cdot \prod_{j=0}^{r-1}\left(1-\sum_{k=0}^j S_{k,M_\mcal{H}}\right). \]
\end{theo}

\begin{theo}\label{theo:U-cobordant}
The complex cobordism class $[W_\mcal{H}(\mbb{C}) \to \mrm{pt}]\in MU^*(\mrm{pt})$ of the wonderful compactification $W_{\mcal{H}}$ of a complex hyperplane arrangement $\mcal{H}$ depends only on the matroid $M_\mcal{H}$ of $\mcal{H}$. Similarly, the algebraic cobordism class $[W_\mcal{H}\to \mrm{pt}]\in \Omega^*(\mrm{pt})$ depends only on the matroid $M_{\mcal{H}}$ of $\mcal{H}$. 
\end{theo}
\begin{proof}
As the ring structure of $\CH^*(W_{\mcal{H}})$ and the total Chern class $c(T_{W_{\mcal{H}}})$ depend only on $M_\mcal{H}$, the Chern numbers of $W_{\mcal{H}}$ depend only on $M_\mcal{H}$. This observation, combined with \Cref{theo:cobordant-manifolds}, yields the claim for $[W_{\mcal{H}}(\mbb{C})\to \mrm{pt}]\in MU^*(\mrm{pt})$. That the claim also holds for $[W_{\mcal{H}}\to \mrm{pt}]\in \Omega^*(\mrm{pt})$ follows from the fact that the $\mbb{L}^*$-algebra homomorphism $\Omega^*(\mrm{pt})\to MU^{2*}(\mrm{pt})$ which sends a class $[X\to \mrm{pt}]\in \Omega^*(\mrm{pt})$ to $[X(\mbb{C})\to \mrm{pt}] \in MU^*(\mrm{pt})$ is an isomorphism \cite[Corollary 1.2.11]{LM07}.
\end{proof}

\begin{rmk}\label{rmk:future-work}
It would be interesting to give an explicit formula for the class $[W_{\mcal{H}}\to \mrm{pt}] \in \Omega^*(\mrm{pt})$ that depends only on the matroid $M_{\mcal{H}}$. It would also be interesting to give a complete description of the proper push-forward map $\Omega^*(W_{\mcal{H}})\to \Omega^*(\mrm{pt})$. Upon doing so, this may generalize to yield an artificial ``cobordism class'' $[X_M]\in \Omega^*(\mrm{pt})$ of a matroid $M$ which agrees with the class $[W_{\mcal{H}}\to \mrm{pt}]\in\Omega^*(\mrm{pt})$ when $M$ is realizable by $\mcal{H}$ over $\mbb{C}$, and also yield an $\mbb{L}^*$-module homomorphism $\Omega^*(X_M) \to \Omega^*(\mrm{pt})$ which agrees with the proper push-forward map $\Omega^*(W_{\mcal{H}})\to \Omega^*(\mrm{pt})$ when $M$ is realizable by $\mcal{H}$ over $\mbb{C}$. A similar program was carried out for $K$-theory in \cite{LLPP24}. An explicit recursive formula for the algebraic cobordism class $[\overline{\mcal{M}_{0,n}} \to \mrm{pt}]$ of the moduli space of stable genus zero curves with $n$ marked points is given in \cite{EB26} in terms of the generators of the Lazard ring $\mbb{L}^*$. The space $\overline{\mcal{M}_{0,n}}$ is the wonderful compactification of a hyperplane arrangement with respect to a non-maximal building set, namely the minimal building set \cite[Section 4.3]{DPwonderful}. Generalizing the techniques developed in \cite{EB26} to give a recursive formula for the algebraic cobordism class of a wonderful variety $[W_\mcal{H}\to\mrm{pt}]$ with respect to a maximal building set is work in progress of the first author.
\end{rmk}

\printbibliography
	
\end{document}